\documentclass[preprint,12pt]{elsarticle}
\usepackage{amssymb}
\usepackage{float}
\usepackage{amsmath}
\usepackage{amssymb}

\usepackage{lineno,hyperref}
\modulolinenumbers[5]

\usepackage{xcolor}
\hypersetup{
    colorlinks,
    linkcolor={blue},
    citecolor={blue},
    urlcolor={blue}
}

\usepackage{lipsum}
\makeatletter
\def\ps@pprintTitle{%
 \let\@oddhead\@empty
 \let\@evenhead\@empty
 \def\@oddfoot{}%
 \let\@evenfoot\@oddfoot}
\makeatother

\newtheorem{thm}{Theorem}
\newtheorem{proposition}[thm]{Proposition}
\newdefinition{remark}{Remark}
\newproof{proof}{Proof}
\begin{document}

\begin{frontmatter}

\title{Fault Detection and Isolation of Satellite Gyroscopes Using Relative Positions in Formation Flying}

\author{Amir Shakouri\fnref{label01}}
\fntext[label01]{Research Assistant, Department of Aerospace Engineering, \href{mailto:a_shakouri@ae.sharif.edu}{a\_shakouri@ae.sharif.edu}}
\author{Nima Assadian\fnref{label03}}
\fntext[label03]{Associate Professor, Department of Aerospace Engineering, \href{mailto:assadian@sharif.edu}{assadian@sharif.edu}}
\address{Sharif University of Technology, 145888 Tehran, Iran}

\begin{abstract}

A fault detection and isolation method for satellite rate gyros is proposed based on using the satellite-to-satellite measurements such as relative position beside orbit parameters of the primary satellite. By finding a constant of motion, it is shown that the dynamic states in a relative motion are restricted in such a way that the angular velocity vector of primary satellite lies on a quadratic surface. This constant of motion is then used to detect the gyroscope faults and estimate the corresponding scale factor or bias values of the rate gyros of the primary satellite. The proposed algorithm works even in time variant fault situations as well, and does not impose any additional subsystems to formation flying satellites. Monte-Carlo simulations are used to ensure that the algorithm retains its performance in the presence of uncertainties. In presence of only measurement noise, the isolation process performs well by selecting a proper threshold. However, the isolation performance degrades as the scale factor approaches unity or bias approaches zero. Finally, the effect of orbital perturbations on isolation process is investigated by including the effect of zonal harmonics as well as drag and without loss of generality, it is shown that the perturbation effects are negligible. 

\end{abstract}


\end{frontmatter}

%\linenumbers

\section{Introduction}
\label{S:1}
 
For the sake of high reliability and safety, spacecraft should tolerate the faults of their subsystems and components. Thus, fault detection and isolation (FDI) and consequently fault recovery algorithms are a part of mission management system on-board or off-board the spacecraft. However, modern space missions require the capability of handling faults with minimum ground support \citep{Tipaldi01}. In a survey by Tafazoli \citep{Tafazoli02} 156 on-orbit failures has been identified from 1980 to 2005 of which 40\% were catastrophic. Attitude and orbit control subsystem (AOCS) caused more mission failures than any other subsystem (32\% of the whole) and gyroscopes are the reason of most AOCS failures (17\%).

FDI methods traditionally can be summarized in three major categories \citep{Ding03}; hardware redundancy based, signal processing based, and plausibility test. Hardware redundancy based FDI is the simplest and the most expensive solution. The high reliability and direct fault isolation are the most mentioned advantages of this method \citep{Pittelkau04}. Nonetheless, there are cases as BeppoSAX or ERS2 that the spacecraft lost primary as well as spare gyroscopes over a period of 5 years \citep{Tafazoli02}. The two other methods are more cost-effective than hardware redundancy. However, their main drawback is the need of high speed onboard computers that can run the fault diagnosis algorithms on-line. Nevertheless, some other objectives or constraints such as robustness, reactive detection, quick isolation, and limited onboard resources (CPU and memory) should be taken into account in the selection of the FDI strategy \citep{Tipaldi01}.

Beside gyroscopes, attitude sensors such as star trackers \citep{Williamson05}, sun sensors and earth sensors \citep{Gao06,Xiong07,Das08} or redundant gyroscopes \citep{Li09} can lead to FDI solutions. Most of these studies utilize different linear and nonlinear filtering approaches. Nonetheless, other approaches such as using conservation of angular momentum are also examined for gyroscope fault detection \citep{Markley10}.

Great advantages of formation-flying (FF) have made it suitable for many space missions of NASA, Department of Defense, ESA and other space agencies \citep{Scharf11}. Reducing the costs and increasing the flexibility of space missions are the most important advantages of using multiple satellites. FF missions can accomplish goals that are impossible or very difficult by a monolithic satellite \citep{Sabol12}; missions such as PRISMA \citep{Gill13}, TanDEM-X \citep{Krieger14}, and TerraSAR-X \citep{Pitz15}.

High-precision requirements in FF control strategies makes FDI more important in this kind of missions. FF satellites can use conventional FDI algorithms with/without utilizing their relative information. Actuator fault estimation in FF has been investigated by various methods. The concept of hierarchical architecture using a cooperative scheme is investigated in \citep{Azizi16}. A dynamic neural network-based method using relative attitudes is presented in \citep{Valdes17}. A hierarchical methodology using neural network-based scheme is investigated in \citep{Valdes18}. Actuator FDI in a network of unmanned vehicles for different architectures is presented in \citep{Meskin19}. Fault tolerant control in FF has been investigated in different approaches. Lee et al. have studied the use of GPS in estimating the relative positioning \citep{Lee20}. The use of RADAR sensor for measuring relative position, azimuth and elevation angle is investigated by Ilyas et al. \citep{Ilyas21}. Thanapalan et al. studied a redundancy based approach \citep{Thanapalan22}.

There exist many different approaches in the relative navigation (RN) of FF satellites. The RN can be done by the use of global positioning system (GPS) \citep{Montenbruck23,Tancredi24} for near earth satellites or GPS-like technologies \citep{Purcell25} for deep space missions. Satellite-to-satellite tracking (SST) methods \citep{Kim26} can be used for RN as well. SST can be attained using different kinds of measurements; range \citep{Keating27,Jian-feng28}, range rate \citep{MacArthur29}, line-of-sight vectors \citep{Patel30,Gaias31}, and combinations of them \citep{Christian32,Decoust33,Wang34}. Prior research also considers the dynamical behavior of satellites in FF including perturbations \citep{Baoyin35,Cho36}.

This paper deals with a novel FDI method that is based on relative equations of motion. In this proof-of-concept study, it is supposed that the relative position of the secondary satellite is measured in the primary satellite body frame. The first and second derivatives of relative position has been computed by a finite difference method of fourth order. A constant of motion is found which is independent of the absolute dynamical states of the secondary satellite. This constant of motion is used as the residual to be utilized for fault detection of the primary satellite gyroscopes. Moreover, some analytical formulas are found using this constant of motion for fault isolation and identification purposes.

The organization of this paper proceeds as follows. First, the constant of motion is derived that relates the rotational motion of primary satellite to its absolute translational motion and the relative dynamics. After that, a sensitivity analysis on the basic equation is presented. Next, fault determination process and the effect of thresholds on the detection of slight faults are analyzed. Next, fault isolation process and the proposed algorithm is described. Then, simulation results based on Monte-Carlo method for two dynamic scenarios and different faults are presented. Finally, the effect of perturbations on fault isolation for scale factors and biases are obtained.

\section{Constant of Motion}
\label{S:2}

Consider two satellites (primary and secondary) flying in two different trajectories around the Earth (Fig. \ref{fig:1}). The relative acceleration of the secondary satellite with respect to the primary satellite frame can be stated as \citep{Goldstein37}
\begin{equation}
\label{eq:1}
\boldsymbol a_S^P=\boldsymbol a_S^O-\boldsymbol a_P^O-\dot{\boldsymbol \omega}^{PE}\times \boldsymbol r_{SP}-2\boldsymbol \omega^{PE}\times \boldsymbol v_S^P-\boldsymbol \omega^{PE}\times(\boldsymbol \omega^{PE}\times \boldsymbol r_{SP})
\end{equation}
where $\boldsymbol a_S^O$ and $\boldsymbol a_P^O$ are the secondary and the primary satellites accelerations in an inertial coordinate systes, respectively. They can be replaced by their universal gravity formulation ($-\mu\boldsymbol r/r^3$) plus perturbation terms. The $\boldsymbol r_{SP}$ is the position vector of the secondary satellite relative to the primary and can be defined and measured in the primary body coordinate system. $\boldsymbol v_S^P$ is the time derivative of $\boldsymbol r_{SP}$ with respect to primary satellite frame. $\boldsymbol \omega^{PE}$ and $\dot{\boldsymbol \omega}^{PE}$ are the angular velocity and acceleration of the primary satellite body with respect to inertial coordinate system, respectively. By defining $\boldsymbol f(t)$, Eq. \eqref{eq:1} can be simplified as follows:
\begin{equation}
\label{eq:2}
\dot{\boldsymbol \omega}^{PE}\times \boldsymbol r_{SP}+2\boldsymbol \omega^{PE}\times \boldsymbol v_S^P+\boldsymbol \omega^{PE}\times(\boldsymbol \omega^{PE}\times \boldsymbol r_{SP})+\boldsymbol f(t)=0
\end{equation}
where
$$\boldsymbol f(t)=\boldsymbol a_S^P+\frac{\mu}{\|\boldsymbol r_{SP}+\boldsymbol r_{PO}\|^3}(\boldsymbol r_{SP}+\boldsymbol r_{PO})-\frac{\mu}{r_{PO}^3}\boldsymbol r_{PO}+\boldsymbol f_p(\boldsymbol r_{PO},\boldsymbol r_{SP})$$

\begin{figure}[H]
\centering\includegraphics[width=0.7\linewidth]{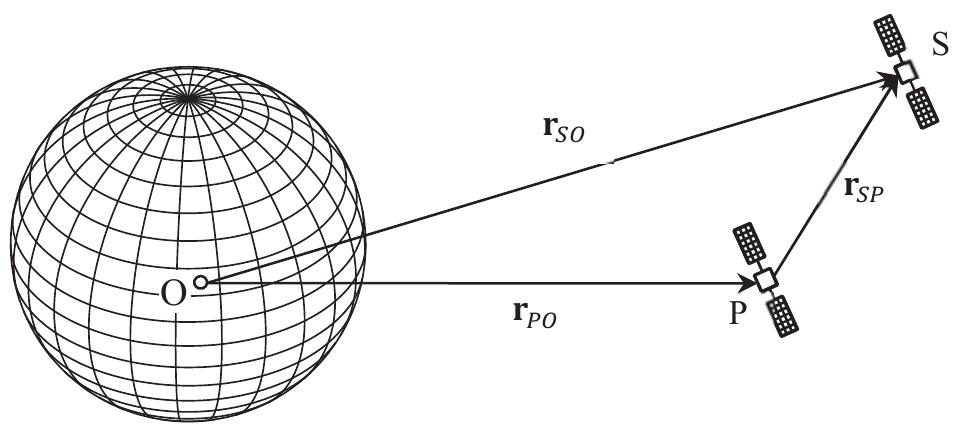}
\caption{Schematic of two orbiting satellites and their relative position.}
\label{fig:1}
\end{figure}

The perturbation term, $\boldsymbol f_p(\boldsymbol r_{PO},\boldsymbol r_{SP})$, is a function of $\boldsymbol r_{PO}$ and $\boldsymbol r_{SP}$ that includes the effect of conservative perturbation accelerations. Effect of non conservative perturbations are ignored here. Section \ref{S:7} studies the effect of any ignored terms (including conservative and non-conservative perturbations) in the function of $\boldsymbol f(t)$. The superscript of $\boldsymbol \omega^{PE}$ is removed for simplicity, i.e., $\boldsymbol \omega^{PE}\equiv\boldsymbol \omega=[\omega_x\quad\omega_y\quad\omega_z]^T$. Multiplying Eq. \eqref{eq:2} by $\boldsymbol r_{SP}^T$ and writing the equation as a function of angular acceleration elements, the following constant of motion is obtainable:
\begin{equation}
\label{eq:3}
\begin{split}
\Psi=A\omega_x^2+B\omega_y^2+C\omega_z^2+2D\omega_x\omega_y+2E\omega_y\omega_z+2F\omega_z\omega_x \\
+2G\omega_x+2H\omega_y+2J\omega_z+K
\end{split}
\end{equation}
Eq. \eqref{eq:3} is a quadratic surface in terms of $\omega_x$, $\omega_y$ and $\omega_z$. Parameters $A$ to $J$ are defined in Table \ref{table:1} and are functions of the relative position and velocity which can be measured or computed. Primary satellite absolute dynamic states are collected in $K$ parameter.

Eq. \eqref{eq:3} can be expressed in matrix form as: 
\begin{equation}
\label{eq:4}
\Psi=W^T\begin{bmatrix}
\mathcal{G} & \boldsymbol\beta \\
\boldsymbol\beta^T & K
\end{bmatrix}W=\boldsymbol \omega^T\mathcal{G}\boldsymbol \omega+2\boldsymbol \beta^T\boldsymbol \omega+K
\end{equation}
in which
$$W=\begin{bmatrix}
\boldsymbol \omega \\
1
\end{bmatrix}\quad\mathcal{G}=\begin{bmatrix}
A & D & F \\
D & B & E \\
F & E & C \\
\end{bmatrix}\quad\boldsymbol \beta=\begin{bmatrix}
G \\
H \\
J \\
\end{bmatrix}$$

\begin{table}[h]
\caption{Definitions of constant of motion parameters}
\centering
\begin{tabular}{ll}
\hline
\hline
Parameter & Definition \\
\hline
$A$ & $-({r_{SP}^2}_y+{r_{SP}^2}_z)={r_{SP}^2}_x-r_{SP}^2$ \\
$B$ & $-({r_{SP}^2}_x+{r_{SP}^2}_z)={r_{SP}^2}_y-r_{SP}^2$ \\
$C$ & $-({r_{SP}^2}_x+{r_{SP}^2}_y)={r_{SP}^2}_z-r_{SP}^2$ \\
$D$ & ${r_{SP}}_y{r_{SP}}_x$ \\
$E$ & ${r_{SP}}_z{r_{SP}}_y$ \\
$F$ & ${r_{SP}}_z{r_{SP}}_x$ \\
$G$ & ${r_{SP}}_z{v_S^P}_y-{r_{SP}}_y{v_S^P}_z$ \\
$H$ & ${r_{SP}}_x{v_S^P}_z-{r_{SP}}_z{v_S^P}_x$ \\
$J$ & ${r_{SP}}_y{v_S^P}_x-{r_{SP}}_x{v_S^P}_y$ \\
$K$ & $\boldsymbol r_{SP}^T \boldsymbol f={r_{SP}}_xf_x+{r_{SP}}_yf_y+{r_{SP}}_zf_z$ \\
\hline
\hline
\end{tabular}
\label{table:1}
\end{table}

The value of the scalar $\Psi$ should be zero. Let us introduce the measured values by adding an accent mark Tilde ($\sim$). However, if the measured values of angular velocity and relative positions are used ($\widetilde{\boldsymbol \omega}=\boldsymbol \omega+\boldsymbol \nu$), this function may have nonzero values due to the measurement noises ($\boldsymbol \nu$). Assuming negligible noises for the absolute and relative positioning, and supposing linear, zero mean, and uncorrelated noise model for the gyroscope measurements, the expected value of Eq. \eqref{eq:4} is:
\begin{equation}
\label{eq:5}
\text{E}{[\widetilde{\Psi}]}=\text{tr}\left\{
\begin{bmatrix}
\mathcal{G} & \boldsymbol\beta \\
\boldsymbol\beta^T & K
\end{bmatrix}
\Sigma_W\right\}
+\overline{W}^T
\begin{bmatrix}
\mathcal{G} & \boldsymbol\beta \\
\boldsymbol\beta^T & K
\end{bmatrix}
\overline{W}
\end{equation}
where
$$
\overline{W}\triangleq\text{E}[\widetilde{W}]=\begin{bmatrix}
\overline{\boldsymbol\omega} \\
1
\end{bmatrix}$$$$\Sigma_W\triangleq\text{E}[(\widetilde{W}-\overline{W})(\widetilde{W}-\overline{W})^T]=\begin{bmatrix}
R & 0 \\
0 & 0
\end{bmatrix}$$$$ R \triangleq\text{E}[\boldsymbol \nu^T\boldsymbol \nu]=\sigma_g^2I
$$
In the above expressions, $R$ is the covariance matrix of gyroscope measurement noise and assumed to be identical for all directions ($\sigma_g^2$). Since a linear, uncorrelated model for the noise is assumed, $\overline{W}=W$. So, Eq. \eqref{eq:5} can be reduced to
\begin{equation}
\label{eq:6}
\text{E}[\widetilde{\Psi}]=-2r_{SP}^2\sigma_g^2
\end{equation}
where $r_{SP}=\|\boldsymbol r_{SP}\|$. From Eq. \eqref{eq:6}, the mean value of constant of motion is negative, independent of orbital elements and rotational dynamics. Again, following the previous assumptions and definitions, the variance of $\widetilde{\Psi}$ can be calculated as
\begin{equation}
\label{eq:7}
\begin{split}
\text{Var}[\widetilde{\Psi}]=2\text{tr}\left\{\begin{bmatrix}
\mathcal{G} & \boldsymbol\beta \\
\boldsymbol\beta^T & K
\end{bmatrix}
\Sigma_W\begin{bmatrix}
\mathcal{G} & \boldsymbol\beta \\
\boldsymbol\beta^T & K
\end{bmatrix}
\Sigma_W\right\}+\\
4\overline{W}^T\begin{bmatrix}
\mathcal{G} & \boldsymbol\beta \\
\boldsymbol\beta^T & K
\end{bmatrix}\Sigma_W\begin{bmatrix}
\mathcal{G} & \boldsymbol\beta \\
\boldsymbol\beta^T & K
\end{bmatrix}\overline{W}
\end{split}
\end{equation}

From Eq. \eqref{eq:7}, the variance of constant of motion is independent of orbit elements but is a function of rotational dynamics. The above equations are derived for the rate gyro measurements noise only. Generally, the analytical relation of mean values and variance of $\widetilde{\Psi}$ cannot be easily derived including the relative position and velocity measurements. However, the simulations show that the value of $\widetilde{\Psi}$ still remains in a bounded region around zero as long as the measurement noises have zero mean values. This beneficial characteristic can be used for the fault detection of any relevant sensor, because any fault in the sensors deviates the function from the bound. In this paper the fault detection of rate gyros is studied.

Furthermore, the function $\Psi$ restricts the angular velocity on a quadric surface at each time instant. So, finding other restricting surfaces (as can be derived from energy or momentum analysis) can lead to an estimation of the angular velocity.

\section{Sensitivity Analysis}
\label{S:3}

Because the $\Psi$ function is used for rate gyros fault detection and isolation, it is analyzed for its sensitivity to angular velocity vector. This analysis show the behavior of $\Psi$ function with respect to deviations in angular velocities. Taking a matrix derivative from Eq. \eqref{eq:4} yields:
\begin{equation}
\label{eq:8}
\frac{1}{2}\frac{\partial\widetilde{\Psi}}{\partial\widetilde{\boldsymbol \omega}}=\mathcal{G}\widetilde{\boldsymbol \omega}+\boldsymbol \beta
\end{equation}

Parameters $A$,...,$K$ that are defined in Table \ref{table:1}, can be written in a different form as functions of direction cosines of $\boldsymbol r_{SP}$ with respect to the body coordinate axes. Applying this change, the following expression for $\mathcal{G}$ and $\boldsymbol\beta$ are obtained: 
\begin{equation}
\label{eq:9}
\mathcal{G}=r_{SP}^2\left\{\begin{bmatrix}
\cos^2\theta_x & \cos\theta_y\cos\theta_x & \cos\theta_z\cos\theta_x \\
\cos\theta_y\cos\theta_x & \cos^2\theta_y & \cos\theta_z\cos\theta_y \\
\cos\theta_z\cos\theta_x & \cos\theta_z\cos\theta_y & \cos^2\theta_z
\end{bmatrix}-I_{3\times3}\right\}
\end{equation}
\begin{equation}
\label{eq:10}
\boldsymbol \beta=\boldsymbol v_S^P\times\boldsymbol r_{SP}
\end{equation}
Eqs. \eqref{eq:9} and \eqref{eq:10} show the sensitivity is a quadratic function of $r_{SP}$.

Simulations for two FF satellites indicate how a failure in $x$-direction gyro changes the $\widetilde{\Psi}$ function for different values of $r_{SP}$. Fig. \ref{fig:2} shows a sample of 10 simulations in 120 seconds with the following details:

Primary satellite elements are defined in Table \ref{table:2} for SPHERES (Synchronized Position Hold, Engage, Reorient, Experimental Satellites). Initial angular velocity vector of primary satellite is selected to be $\boldsymbol\omega=[3\quad2.5\quad5]^T$ $\text{deg/s}$ and the initial attitude is laid on the primary satellite RSW coordinate system. The RSW coordinate is defined such that its $x$-axis is in the direction of the position vector of the satellite, the $z$-axis towards the orbital angular momentum vector, and the $y$-axis completes the right-handed coordinate system. Secondary satellite orbital elements are same as the primary satellite only the semi-major axis is different. This difference varies from $0.1\text{ km}$ to $1\text{ km}$ in steps of $0.1\text{ km}$ as indicated in Fig. \ref{fig:2}. It is supposed that a fault has occurred in $t=23\text{ s}$ that caused the $x$-gyro to measure 50\% of the real angular velocity of the $x$-axis.

\begin{table}[h]
\caption{A typical scenario for primary satellites \citep{Mohan38}.}
\centering
\begin{tabular}{lll}
\hline
\hline
Parameter & Unit & Value \\
\hline
\multicolumn{3}{c}{\textit{Orbit Elements}} \\
$a$ & km & $6783.34174$ \\
$e$ & -- & $0.0014021$ \\
$i$ & deg & $51.27632$ \\
$\omega$ & deg & $90.69731$ \\
$\Omega$ & deg & $275.17058$ \\
$\nu$ & deg & $309.67626$ \\
\multicolumn{3}{c}{\textit{Moments of Inertia}} \\
$I_{xx}$ & $\text{kg}\cdot\text{m}^2$ & $2.29\times10^{-2}$ \\
$I_{yy}$ & $\text{kg}\cdot\text{m}^2$ & $2.42\times10^{-2}$ \\
$I_{zz}$ & $\text{kg}\cdot\text{m}^2$ & $2.14\times10^{-2}$ \\
\hline
\hline
\end{tabular}
\label{table:2}
\end{table}

\begin{figure}[H]
\centering\includegraphics[width=1\linewidth]{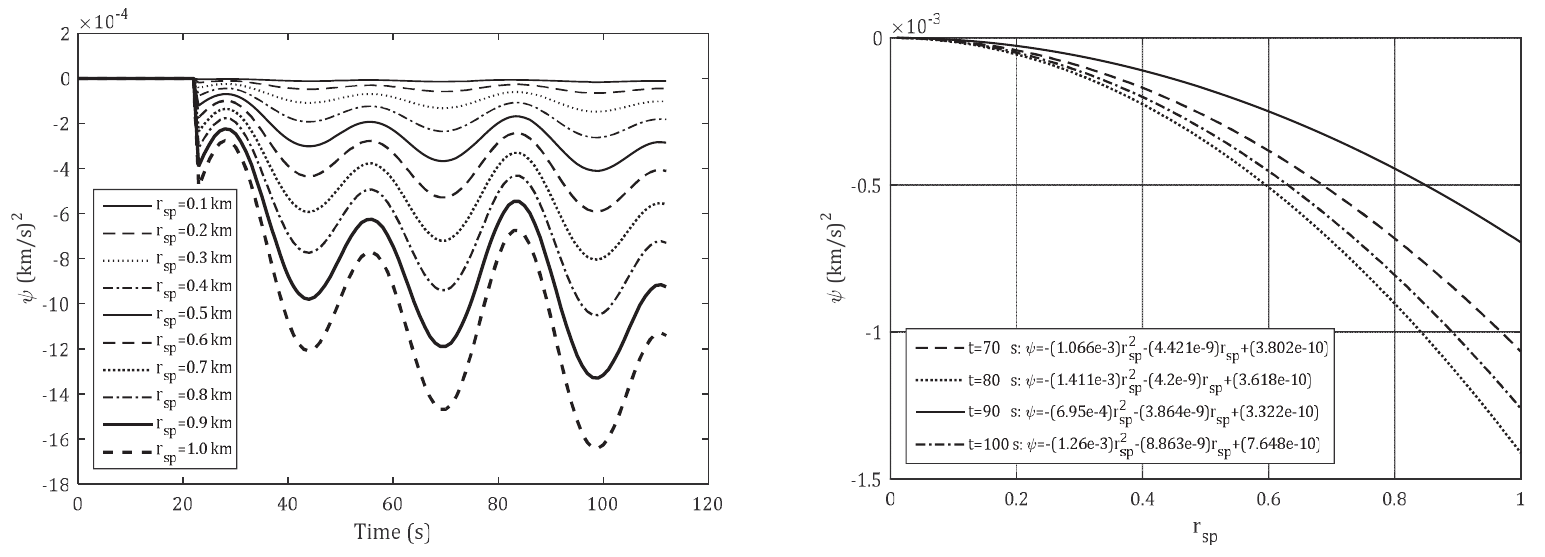}
\caption{Quadratic relation of $\widetilde{\Psi}$ as a function of $\boldsymbol r_{SP}$ in presence of fault.}
\label{fig:2}
\end{figure}

\section{Fault Detection}
\label{S:4}

A threshold should be determined for the fault detection procedure. Selecting a large threshold may prevent the detection of slight faults. Thus, the effect of the specified threshold on the detection of the slight faults should be studied. In this study, the fault in gyros is modeled through scale factor and bias. First, consider a perturbed $\widetilde{\Psi}$ function due to a fault in $x$-gyro ($\widetilde{\omega}_x=s_x\omega_x+b_x$) assuming other gyros are operating correctly ($\widetilde{\omega}_y=\omega_y$, $\widetilde{\omega}_z=\omega_z$), i.e., though all gyros may have bias in practical cases, but the bias of one faulty gyro is larger than the others over the time period of analysis. Hence, Eq. \eqref{eq:3} becomes
\begin{equation}
\label{eq:11}
\begin{split}
\widetilde{\Psi}=A\widetilde{\omega}_x^2+B\omega_y^2+C\omega_z^2+2D\widetilde{\omega}_x\omega_y+2E\omega_y\omega_z+2F\omega_z\widetilde{\omega}_x \\
+2G\widetilde{\omega}_x+2H\omega_y+2J\omega_z+K
\end{split}
\end{equation}

Since the natural variation of angular velocities does not deviate the function $\Psi$ from zero, the partial derivative of the function $\Psi$ with respect to the real values of the angular velocity elements should be zero:
\begin{equation}
\label{eq:12}
\frac{\partial\Psi}{\partial\omega_x}=0\Rightarrow A\omega_x+D\omega_y+F\omega_z+G=0
\end{equation}

Subtracting Eq. \eqref{eq:3} from Eq. \eqref{eq:11}, and using Eq. \eqref{eq:12}, the $\widetilde{\Psi}$ function can be simplified as follows:
\begin{equation}
\label{eq:13}
\widetilde{\Psi}=A(\widetilde{\omega}_x-\omega_x)^2
\end{equation}
and similarly, for two other axes:
\begin{equation}
\label{eq:14}
\widetilde{\Psi}=B(\widetilde{\omega}_y-\omega_y)^2
\end{equation}
\begin{equation}
\label{eq:15}
\widetilde{\Psi}=C(\widetilde{\omega}_z-\omega_z)^2
\end{equation}

Note that each of the above equations are applicable only if the corresponding gyro is defected (by a scale factor or bias) and the others are correct. Therefore, three possibilities exist for the fault occurrence; $x$, $y$, or $z$ direction gyro. For each case, the source of fault can be a bias or a scale factor. Thus, there are six different possibilities that should be analyzed.

It should be also noted that the value of $\widetilde{\Psi}$ in Eqs. \eqref{eq:13} to \eqref{eq:15} can only be non-positive due to any kind of faults in gyros. This is because the coefficients $A$, $B$, and $C$ are nonpositive quantities (Table \ref{table:1}). Moreover, if the secondary satellite lies in the $x$, $y$, or $z$-axis of the primary satellite body coordinate system, the values of $A$, $B$, or $C$ would be zero respectively, and the corresponding gyro fault does not influence the $\widetilde{\Psi}$ function.

\subsection{Scale Factor Analysis}

First, assume that the $x$-gyro has a scale factor error, $s_x$. Therefore, $\widetilde{\omega}_x=s_x\omega_x$ and then:
\begin{equation}
\label{eq:16}
\widetilde{\Psi}=A\omega_x^2(s_x-1)^2
\end{equation}
similarly, for other axes: 
\begin{equation}
\label{eq:17}
\widetilde{\Psi}=B\omega_y^2(s_y-1)^2
\end{equation}
\begin{equation}
\label{eq:18}
\widetilde{\Psi}=C\omega_z^2(s_z-1)^2
\end{equation}

Eqs. \eqref{eq:16} to \eqref{eq:18} indicate that $\widetilde{\Psi}$ has a quadratic behavior with respect to $s_{x,y, \text{ or } z}$, that its maximum value is for $s_{x,y, \text{ or } z}=1$. Fig. \ref{fig:3}-a shows this relation for the simulations with the same details of previous section sample mission.
 
Similar to the previous example, the primary satellite elements are defined in Table \ref{table:2} and initial angular velocity vector of primary satellite is $\boldsymbol\omega=[3\quad2.5\quad5]^T\text{ deg/s}$ and the initial attitude is laid on RSW coordinate system. All secondary satellite orbital elements are the same as the primary satellite except the semi-major axis which is $1\text{ km}$ higher. Fig. \ref{fig:3}-a is plotted for a scale factor fault in $x$-gyro taking place at $t=23\text{ s}$ for $s_x$ from $0$ to $2$.

\subsection{Bias Analysis}

By supposing a bias in $x$-gyro, $b_x$, is the source of the fault. So, $\widetilde{\omega}_x=\omega_x+b_x$, and substituting $\widetilde{\omega}_x$ in Eq. \eqref{eq:13} yields
\begin{equation}
\label{eq:19}
\widetilde{\Psi}=Ab_x^2
\end{equation}
and similarly for other axes: 
\begin{equation}
\label{eq:20}
\widetilde{\Psi}=Bb_y^2
\end{equation}
\begin{equation}
\label{eq:21}
\widetilde{\Psi}=Cb_z^2
\end{equation}

Eqs. \eqref{eq:19} to \eqref{eq:21} show the quadratic relation of the $\widetilde{\Psi}$ as a function of $s_{x,y, \text{ or } z}$ with a maximum value in $b_{x,y, \text{ or } z}=0$. Fig. \ref{fig:3}-b shows this relation with the following details:

Primary and secondary satellite dynamics are as defined in the previous scenario. A failure in $x$-gyro occurs at $t=23\text{ s}$ by a bias that varies from $-1$ to $1\text{ deg/s}$.

The six Eqs. \eqref{eq:16} to \eqref{eq:21} can help to know how a fault can change $\widetilde{\Psi}$ and lead it to cross the specified threshold. Fig. \ref{fig:3}-a and \ref{fig:3}-b show the typical behavior of $\widetilde{\Psi}$ as a function of scale factor and bias of gyro measurements. It can be easily observed from these graphs that any selected threshold results in neglecting slight faults.

\begin{figure}[H]
\centering\includegraphics[width=1\linewidth]{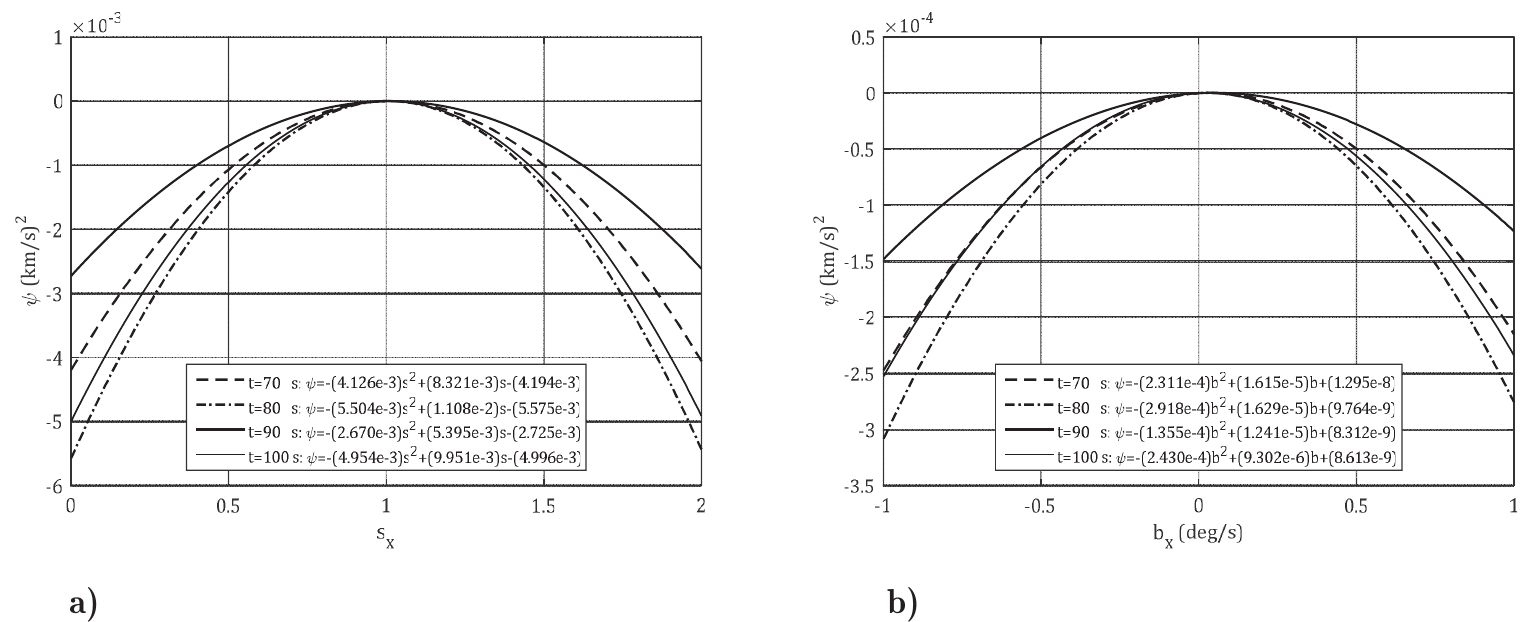}
\caption{Quadratic relation of $\widetilde{\Psi}$ as a function of the a) scale factor $s_x$, and b) bias $b_x$.}
\label{fig:3}
\end{figure}

\section{Fault Isolation}
\label{S:5}

So far, there are six useful equations in hand that shows how a fault (scale factor or bias) can perturb the $\widetilde{\Psi}$ function. Suppose the only faulty sensor is the $x$-gyro and it is defective by a scale factor. Considering Eq. \eqref{eq:13} and substituting $\omega_x=\widetilde{\omega}_x/s_x$, a quadratic equation for $s_x$ is obtained as a function of measured angular velocity $\widetilde{\omega}_x$:
\begin{equation}
\label{eq:22}
(\widetilde{\Psi}-A\widetilde{\omega}_x^2)s_x^2+2A\widetilde{\omega}_x^2s_x-A\widetilde{\omega}_x^2=0
\end{equation}
similarly,
\begin{equation}
\label{eq:23}
(\widetilde{\Psi}-B\widetilde{\omega}_y^2)s_y^2+2B\widetilde{\omega}_y^2s_y-B\widetilde{\omega}_y^2=0
\end{equation}
\begin{equation}
\label{eq:24}
(\widetilde{\Psi}-C\widetilde{\omega}_z^2)s_z^2+2C\widetilde{\omega}_z^2s_z-C\widetilde{\omega}_z^2=0
\end{equation}

The equation of $\widetilde{\Psi}$ for biased gyros are only a quadratic function of bias and it is not a function of angular velocity (Eqs. \eqref{eq:19}-\eqref{eq:21}). In each of these six equations \eqref{eq:19}-\eqref{eq:21}, only one unknown exists which is a scale factor or a bias in a certain direction. Solving Eqs. \eqref{eq:19} to \eqref{eq:24} leads to final equations that can be used for the fault isolation purpose:
\begin{equation}
\label{eq:25}
s_x=\frac{1}{(1\pm\widetilde{\Psi}_x)}, \quad s_y=\frac{1}{(1\pm\widetilde{\Psi}_y)}, \quad s_z=\frac{1}{(1\pm\widetilde{\Psi}_z)}
\end{equation}
\begin{equation}
\label{eq:26}
b_x=\pm\sqrt{\frac{\widetilde{\Psi}}{A}}, \quad b_y=\pm\sqrt{\frac{\widetilde{\Psi}}{B}}, \quad b_z=\pm\sqrt{\frac{\widetilde{\Psi}}{C}}
\end{equation}
where
$$\widetilde{\Psi}_x=\sqrt{\frac{\widetilde{\Psi}}{A\widetilde{\omega}_x^2}}, \quad \widetilde{\Psi}_y=\sqrt{\frac{\widetilde{\Psi}}{B\widetilde{\omega}_y^2}}, \quad \widetilde{\Psi}_z=\sqrt{\frac{\widetilde{\Psi}}{C\widetilde{\omega}_z^2}}$$

Eqs. \eqref{eq:25} and \eqref{eq:26} are used for estimating the values of scale factors or biases which are the sources of faults in this study. It should be noted that these equations work if and only if one source of fault is active (solely scale factor or bias of one gyro). The schematic flowchart of the proposed algorithm is shown in Fig. \ref{fig:4}. This diagram shows the process of fault detection and isolation based on above equations.

The ``decision making'' process in Fig. \ref{fig:4} is for deciding which of six possibilities has been occurred. The six estimated values based on Eqs. \eqref{eq:25} and \eqref{eq:26} are utilized to calculate the six recovered angular velocities, $\widehat{\omega}$, and consequently six recovered functions of $\Psi$ denoted by $\widehat{\Psi}$. The recovered $\widehat{\Psi}$ function of the actual possibility remains in a threshold near zero unlike the others. This is based on the following proposition which is the fundamental of the proposed fault isolation procedure.

\begin{figure}[H]
\centering\includegraphics[width=1\linewidth]{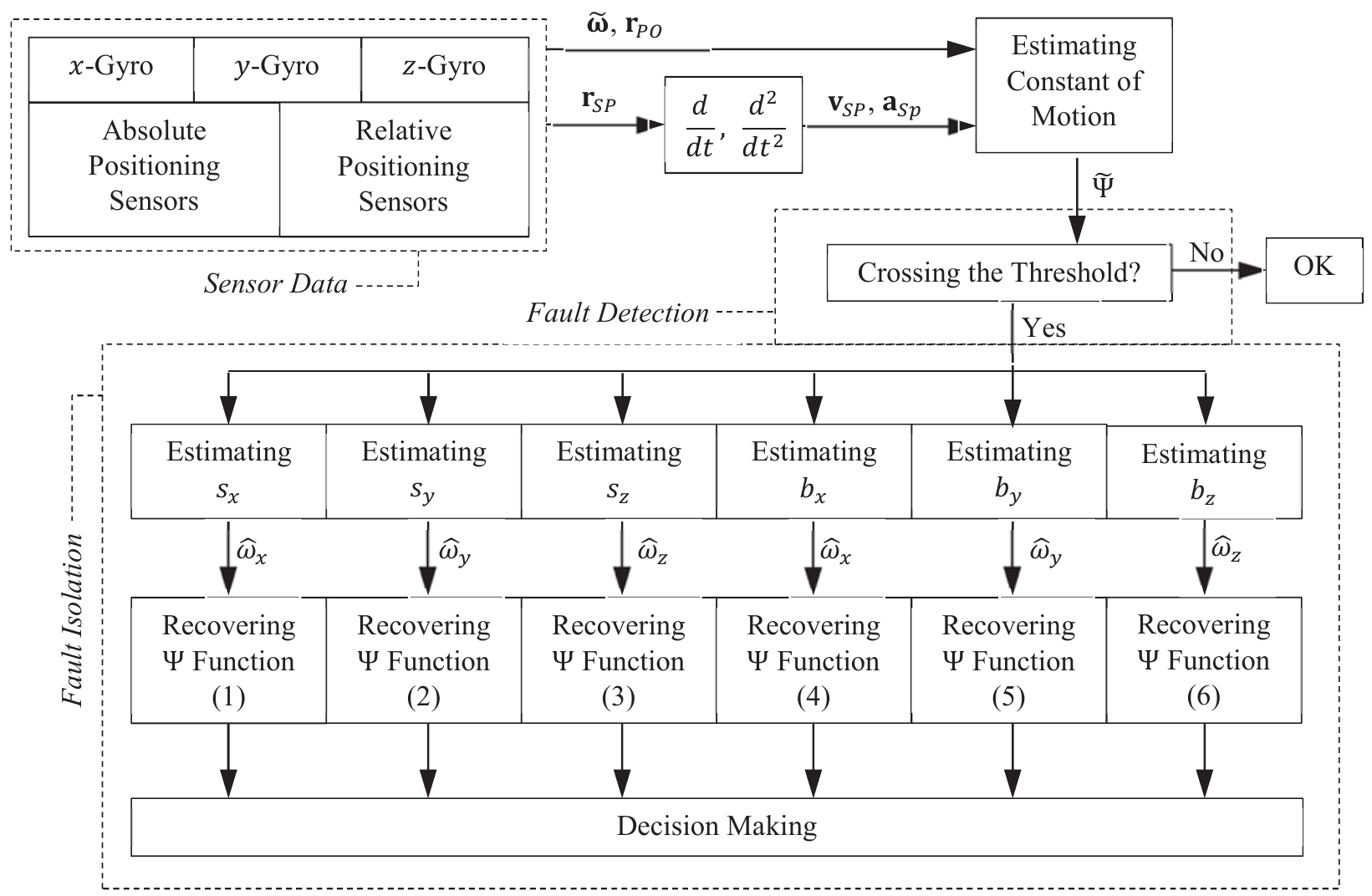}
\caption{The proposed algorithm for the fault detection and isolation.}
\label{fig:4}
\end{figure}

\begin{proposition}
\label{prop:1}
Let $\widetilde{\omega}_x\neq\omega_x$, $\widetilde{\omega}_y=\omega_y$, and $\widetilde{\omega}_z=\omega_z$. If $\widehat{\Psi}=0$, then $\widehat{\omega}_x=\omega_x$ and this is true for the rest of directions.
\end{proposition}
\begin{proof}
If $\widetilde{\omega}_x\neq\omega_x$, $\widetilde{\omega}_y=\omega_y$, and $\widetilde{\omega}_z=\omega_z$, 
$$
\widehat{\Psi}=A\widehat{\omega}_x^2+B\omega_y^2+C\omega_z^2+2D\widehat{\omega}_x\omega_y+2E\omega_y\omega_z+2F\omega_z\widehat{\omega}_x+2G\widehat{\omega}_x+2H\omega_y+2J\omega_z+K
$$
Subtracting Eq. \eqref{eq:3} from above equation yields
$$
\widehat{\Psi}=(\widehat{\omega}_x-\omega_x)[A(\widehat{\omega}_x+\omega_x)+2D\omega_y+2F\omega_z+2G]
$$
For satisfying $\widehat{\Psi}=0$ in a motion with nonzero coefficients (which is generally the case), the above equations leads to $\widehat{\omega}_x-\omega_x=0$. Therefore, if $\widehat{\Psi}=0\Rightarrow\widehat{\omega}=\omega_x$.
\qed
\end{proof}

\begin{remark}
\label{rem:1}
In the case that only one angular velocity deviates from its true value, if the recovered $\widehat{\Psi}$ function approaches zero, the recovered angular velocity approaches the real angular velocity in each direction. It should be noted that Proposition \ref{prop:1} is valid as long as the $\Psi$ function has at least one nonzero coefficient (Table \ref{table:1}), which means that it has nonzero angular velocity and nonzero relative distances and velocities.
\end{remark}

\section{Simulation Results}
\label{S:6}

Two scenarios are defined as fine and coarse cases (scenario (I) and (II), respectively):

\textbf{Scenario (I):} Primary satellite elements are defined in Table \ref{table:1}. Initial angular velocity vector of primary satellite is selected to be $\boldsymbol\omega_0= [3\quad2.5\quad5]^T$ $\text{deg/s}$ and the initial body coordinate system is supposed to be on RSW coordinate system. Secondary satellite orbital elements are same as the primary satellite except of semi-major axis which is $1\text{ km}$ higher. The results of this scenario are promising, that is why it is called the fine scenario.

\textbf{Scenario (II):} All parameters are same as defined in scenario (I) except the initial angular velocity of the primary satellite which is $\boldsymbol\omega_0= [3\quad-2.5\quad15]^T \text{ deg/s}$. As it is shown later, the near zero of $B$ parameter during this scenario results in course estimation of faults. 

The $\Psi$ function in a faultless situation for both scenarios are analyzed. The system is analyzed using a standard Monte-Carlo simulation considering the initial value uncertainties \citep{Asmussen39}. Fig. \ref{fig:5}-a and b indicates the mean value and the standard deviation of $\Psi$ in $1000$ runs with initial uncertainties as given in Table \ref{table:3}. A laser ranging system similar to \citep{Psiaki40} is adopted for the purpose of this study. Psiaki used a ranging accuracy of $0.1\text{ m}$ for $100\text{ km}$ distance, so in this study an accuracy of $0.001\text{ m}$ for $1\text{ km}$ relative distance is selected. Moreover, the angular accuracy of the relative distance ranging is selected to be $0.2\text{ arcsec}$, exactly similar to \citep{Psiaki40}, which results in relative position error of $0.001\text{ m}$ in each axis. The accuracy of angular velocity measurements using a fiber optic gyro (FOG) is selected to be $3.6\text{ deg/hr}$, which is above the standard applicable FOGs \citep{Armenise41}. The relative velocity and acceleration are calculated using a fourth order numerical differentiation method \citep{Gerald42}. While the equations are suitable to take into account the effect of conservative perturbations, the following simulations are in a two-body gravity field. The effect of neglected perturbations are analyzed in the Section \ref{S:7}.

Fig. \ref{fig:5} indicates that the constant of motion remains near zero in the presence of uncertainties and sensor noises. The boundary of the $\Psi$ function is plotted in this figure by $\pm1\sigma$ (one standard deviation away) of the mean value.

\begin{table}[h]
\caption{Standard deviation of the initial parameters for scenario (I) and (II).}
\centering
\begin{tabular}{lll}
\hline
\hline
Parameter & $\sigma$ & Unit \\
\hline
\multicolumn{3}{c}{\textit{Orbit Elements}} \\
$a$ & km & $1.5\times10^1$ \\
$e$ & -- & $1.0\times10^{-5}$ \\
$i$ & deg & $1.0\times10^{-3}$ \\
$\omega$ & deg & $1.0\times10^{-3}$ \\
$\Omega$ & deg & $1.0\times10^{-3}$ \\
$\nu$ & deg & $1.0\times10^{-3}$ \\
\multicolumn{3}{c}{\textit{Initial Angular Velocities}} \\
$\omega_x$,$\omega_y$,$\omega_z$ & $\text{deg/h}$ & $360$ \\
\hline
\hline
\end{tabular}
\label{table:3}
\end{table}

\begin{figure}[H]
\centering\includegraphics[width=1\linewidth]{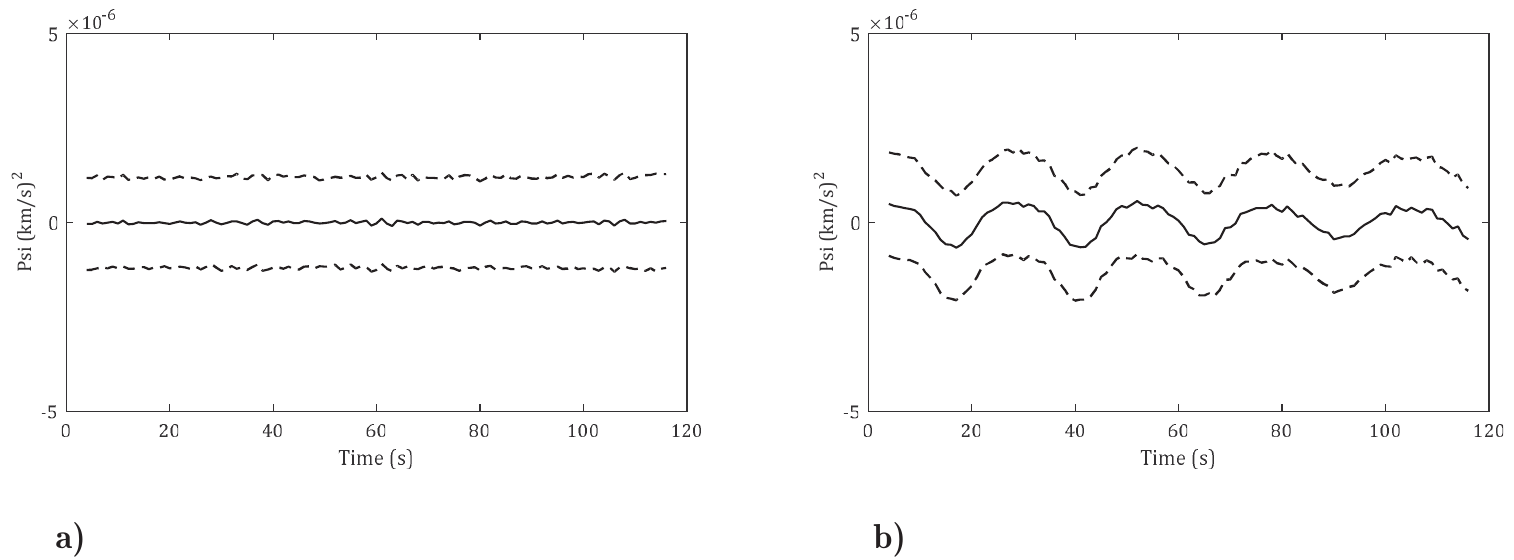}
\caption{The $\Psi$ function mean value (solid line) with $\pm\sigma$ (dashed line) in a faultless noisy condition. a) scenario (I); b) scenario (II).}
\label{fig:5}
\end{figure}

Fig. \ref{fig:6}-a and b show the scatter of $10^3$ runs altogether for both scenarios. This graph can be used for selecting a thresholds above which the $\Psi$ is deviated from zero as a result of other factors than noise. In this way, for scenario (I) and (II) thresholds of $\pm6\times10^{-6}$ and $\pm6.5\times10^{-6}$ has been selected, respectively.

\begin{figure}[H]
\centering\includegraphics[width=1\linewidth]{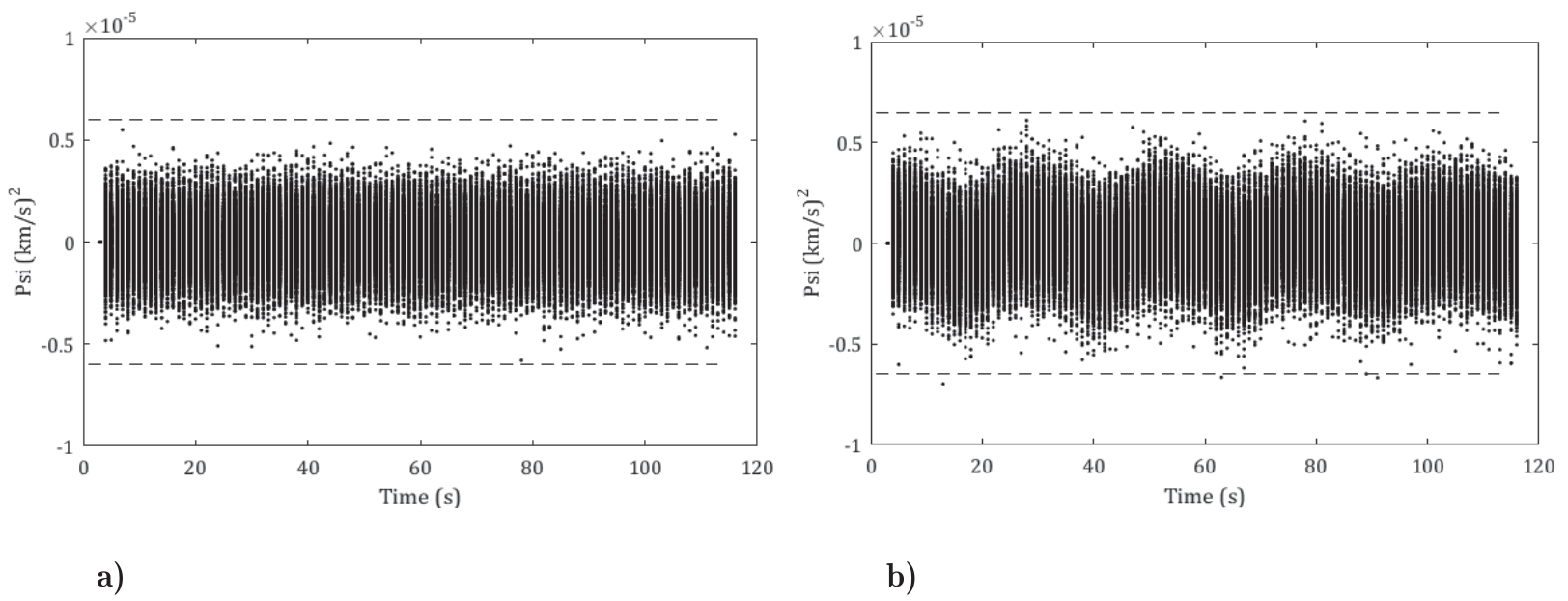}
\caption{The $\Psi$ function in $1000$ faultless noisy conditions and the selected threshold (dashed line) a) scenario (I); b) scenario (II).}
\label{fig:6}
\end{figure}

\subsection{Scale Factor Faults}

Suppose a scale factor of $0.5$ on $x$-gyroscope of scenario (I) and $y$-gyroscope of scenario (II) both are activated in $t=56\text{ s}$. The mean value for the time that fault has been detected for scenario (I) is $t_{FD1}=58.02\text{ s}$ and its standard deviation is $\sigma_{FD1}=0.18\text{ s}$. It means that the fault is detected in around $2$ seconds after occurrence. These values are calculated after $1000$ simulations in random initial condition and in presence of noise. For scenario (II), the mean fault detection time is $t_{FD1}=57.96\text{ s}$ with a standard deviation of $\sigma_{FD1}=1.30\text{ s}$. Thus, in this scenario the faults are generally detected in less than $4$ seconds.

Assuming $s<1$, six possibilities for each scenario exist after fault detection. The estimation corresponding to each possibility are plotted in Figs. \ref{fig:7} and \ref{fig:8} for scenarios I and II, respectively.

\begin{figure}[H]
\centering\includegraphics[width=1\linewidth]{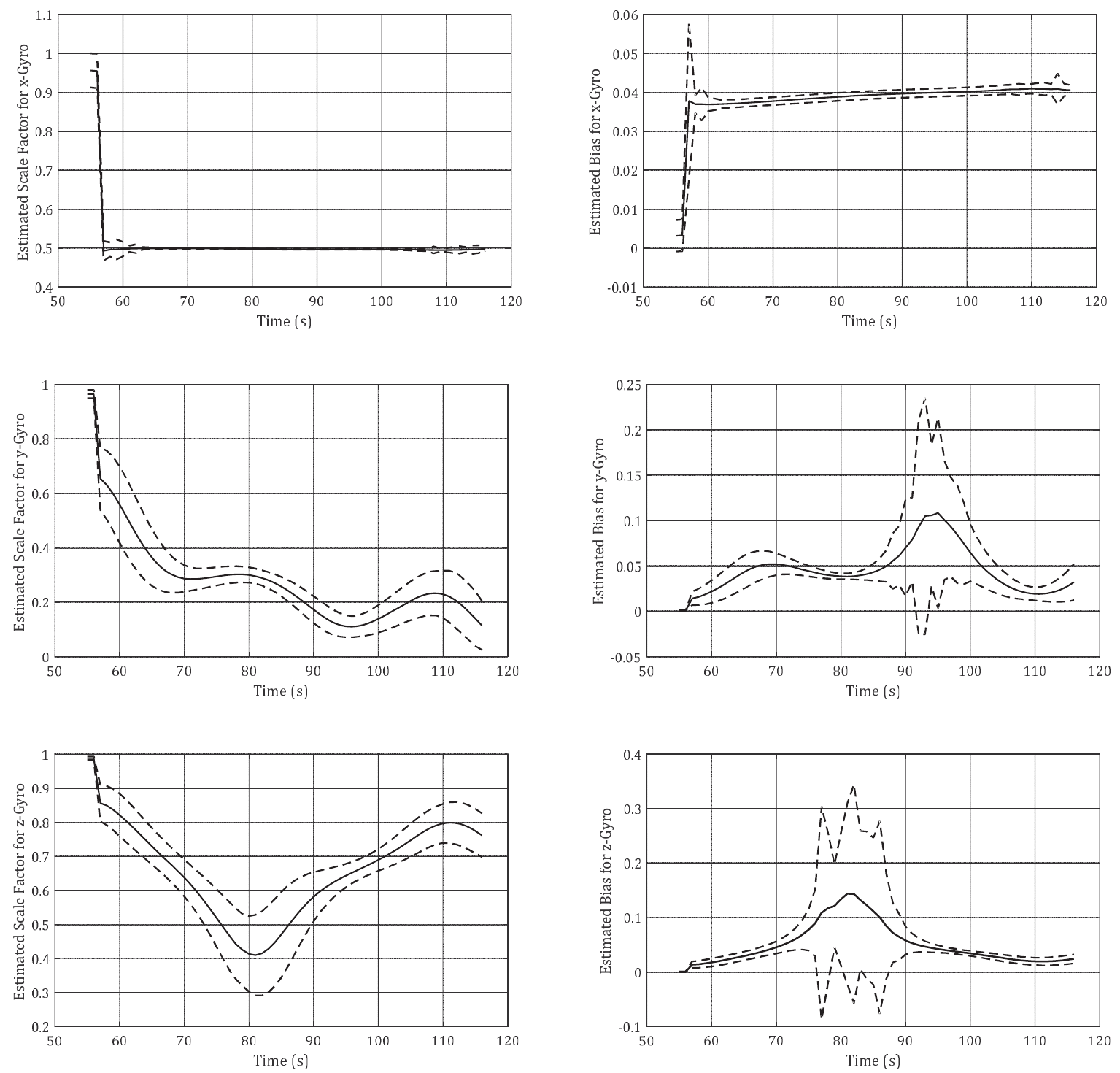}
\caption{Estimated value of six possibilities (scale factor/bias of each gyroscope) for scenario (I).}
\label{fig:7}
\end{figure}

\begin{figure}[H]
\centering\includegraphics[width=1\linewidth]{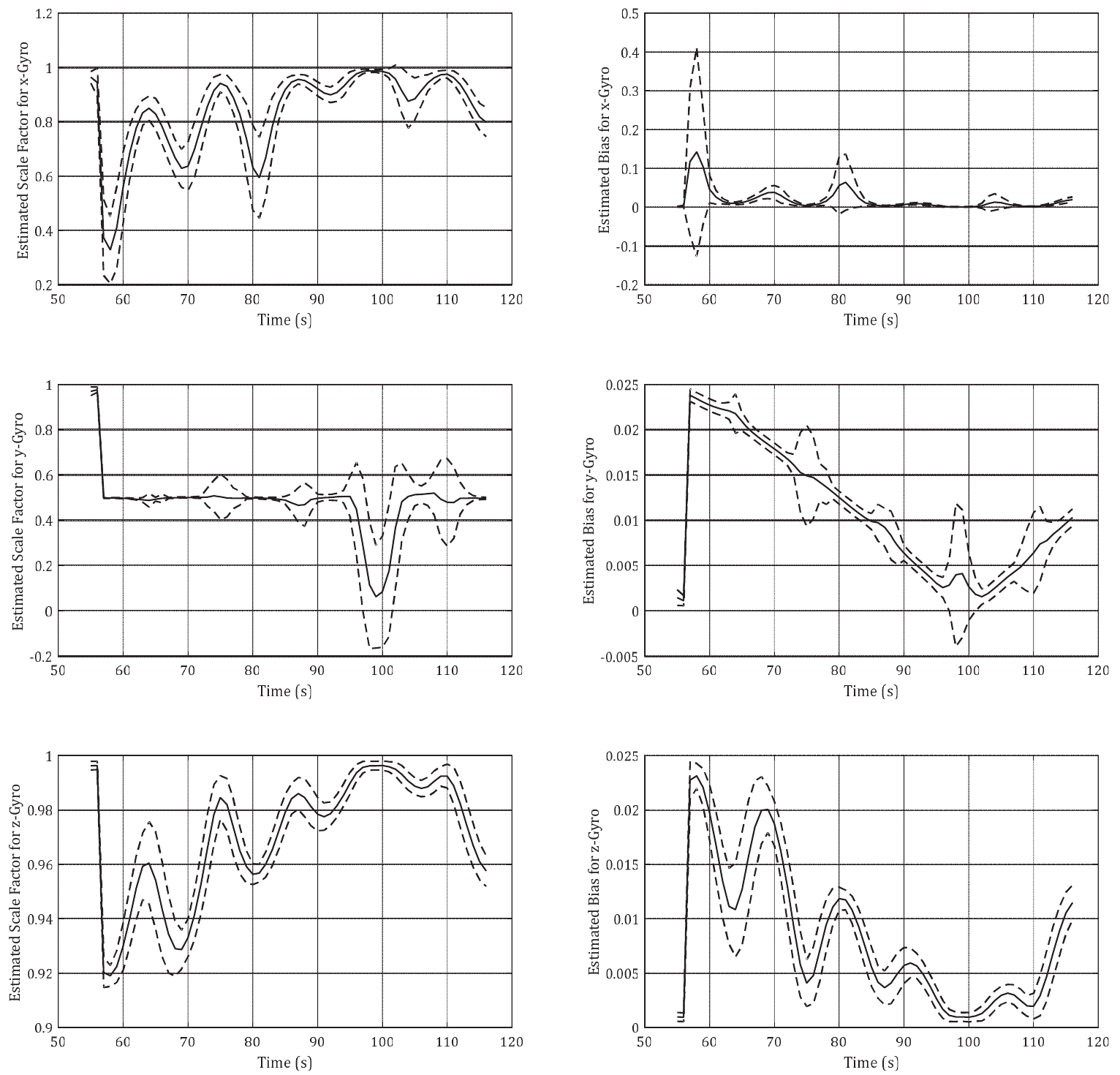}
\caption{Estimated value of six possibilities (scale factor/bias of each gyroscope) for scenario (II).}
\label{fig:8}
\end{figure}

As already stated, a recovery process is required for finding which possibility is correct. These six estimated values for each scenario have to be used in recovering gyroscope outputs and subsequently recovering the $\Psi$ function. Therefore, six recovered $\Psi$ functions exist for each scenario that are plotted in Figs. \ref{fig:9} and \ref{fig:10} for scenarios I and II, respectively.

\begin{figure}[H]
\centering\includegraphics[width=1\linewidth]{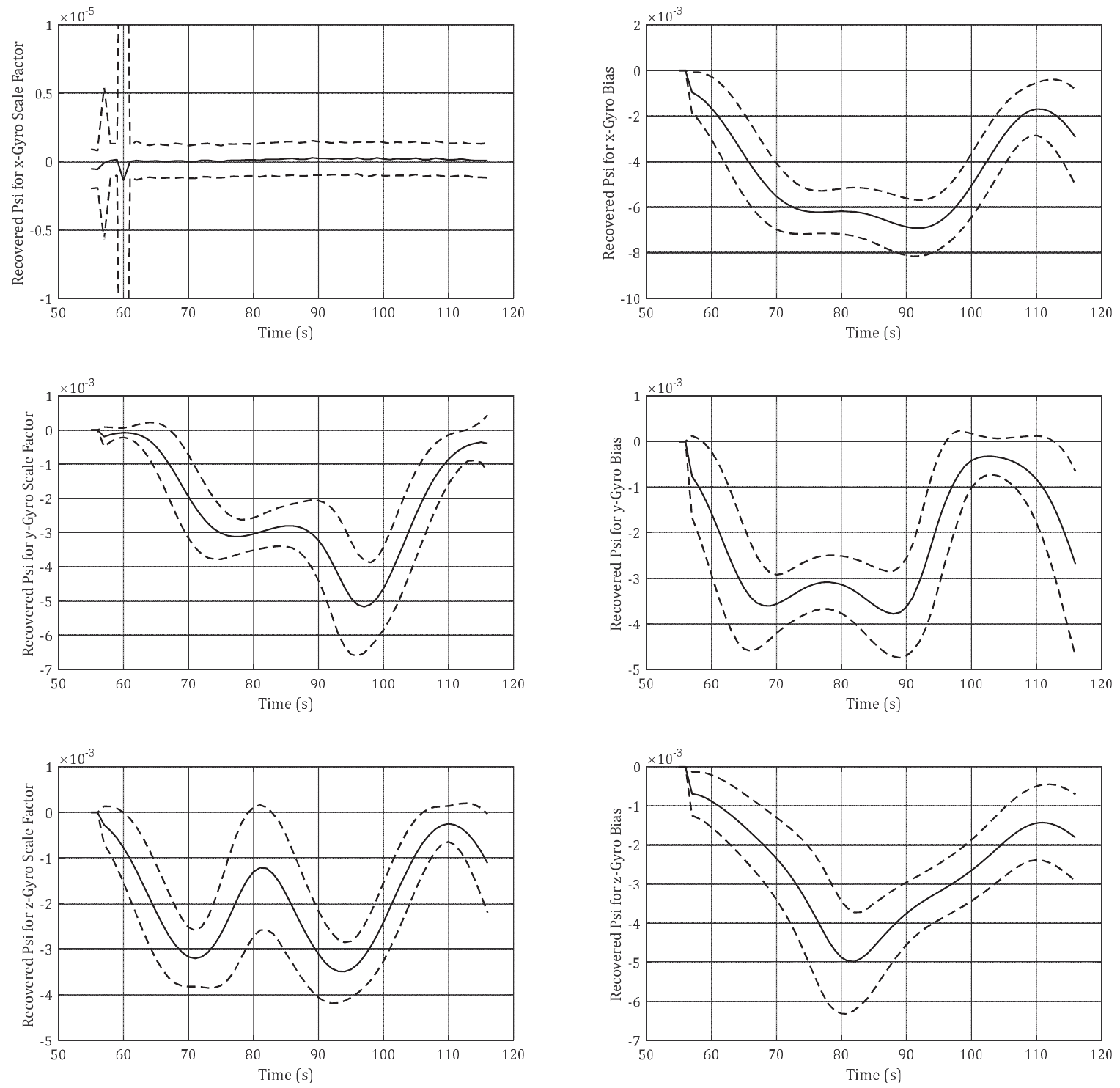}
\caption{Recovered $\Psi$ function of each possibility (scale factor bias of each gyroscope) for scenario (I).}
\label{fig:9}
\end{figure}

\begin{figure}[H]
\centering\includegraphics[width=1\linewidth]{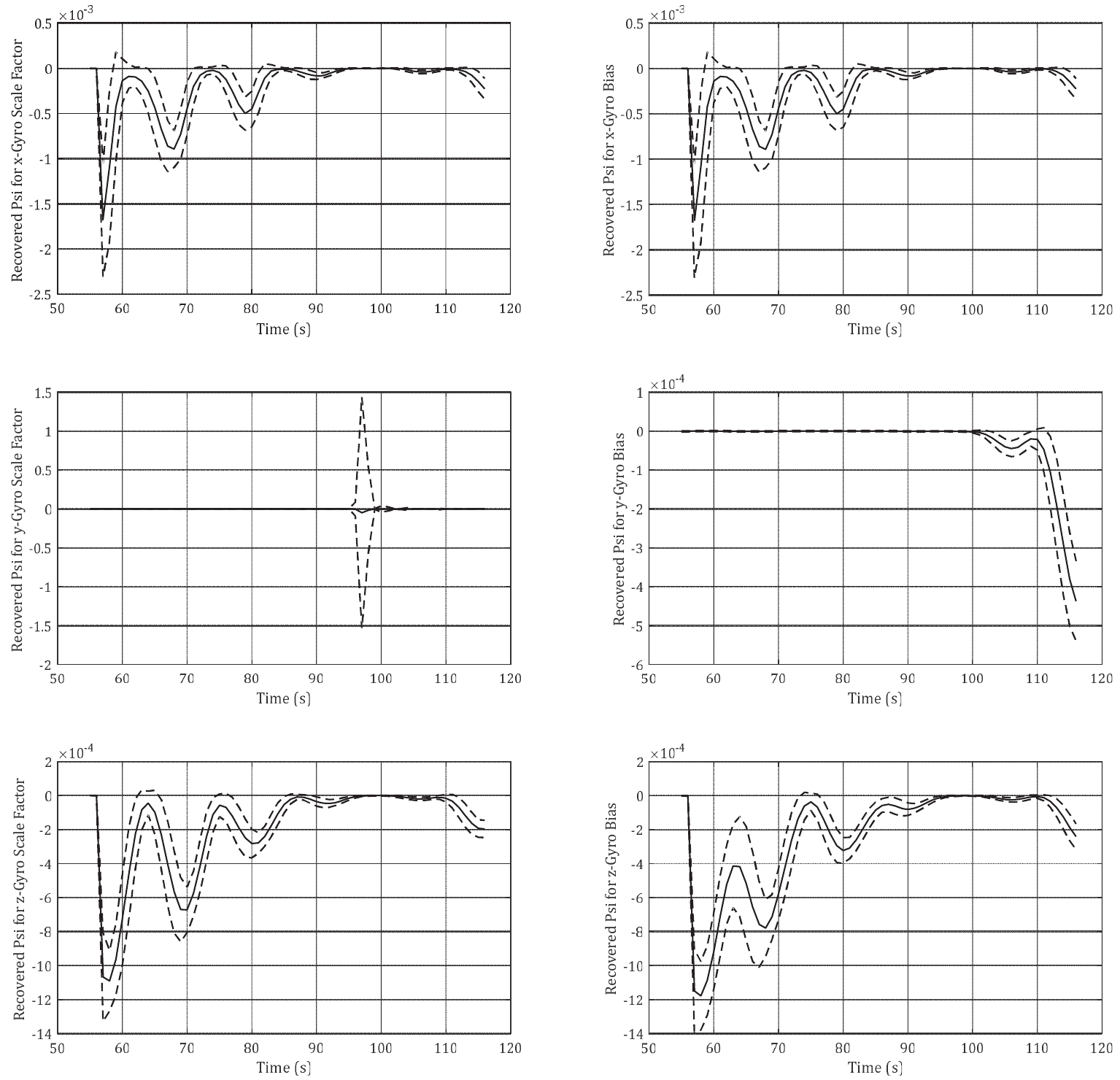}
\caption{Recovered $\Psi$ function of each possibility (scale factor bias of each gyroscope) for scenario (II).}
\label{fig:10}
\end{figure}

Fig. \ref{fig:9} can be used easily for a fault isolation process as the recovered $\Psi$ in $x$-gyroscope scale factor remains in the specified threshold and other recovered functions cross it. Fig. \ref{fig:10} shows fine recovery functions of $\Psi$ for $y$-gyroscope scale factor except of singularity points that $\omega_y$ or parameter $B$ approaches zero (around $t=98\text{ s}$). Fig. \ref{fig:11} shows $\omega_y$ (true value) and parameter $B$ (estimated value) for scenario (II) (the dashed lines are for $\pm1\sigma$ of the mean value).

\begin{figure}[H]
\centering\includegraphics[width=1\linewidth]{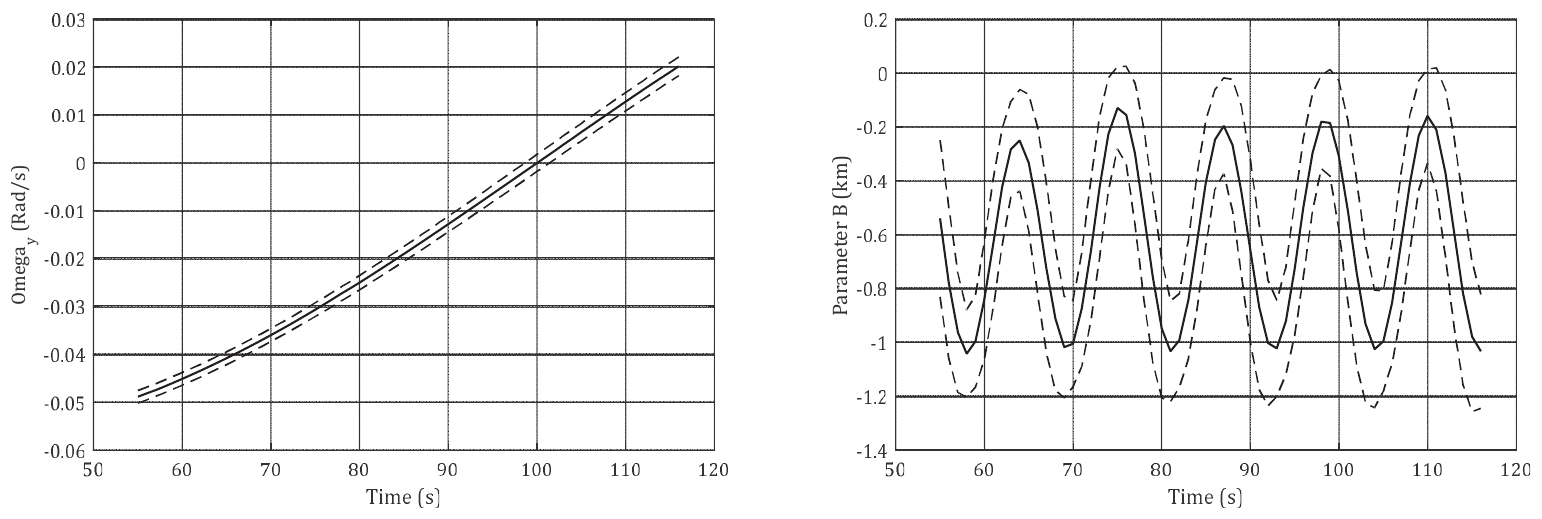}
\caption{$\omega_y$ and the parameter $B$ for scenario (II).}
\label{fig:11}
\end{figure}

It is noteworthy that the recovered $\Psi$ function in $y$-gyroscope bias of Fig. \ref{fig:11} is within the threshold for a period of 45 seconds (until $t=100\text{ s}$). Since the proposed FDI algorithm can work for time varying scale factors or biases, the estimated bias in $y$-gyroscope can be well used in a recovery process of angular velocities. This fact is already stated in Proposition \ref{prop:1} and Remark \ref{rem:1}. Fig. \ref{fig:12} shows the recovered $\omega_y$, supposing faults in scale factor (Fig. \ref{fig:12}-a) and bias (Fig. \ref{fig:12}-b). Fig. \ref{fig:13} shows the difference of two recovered values from the true values. Similar to previous graphs, the dashed line is for $\pm1\sigma$ of the mean value.

According to Fig. \ref{fig:10} it is predictable that for the possibility of $y$-gyroscope bias, estimation fails after $t\simeq 100\text{ s}$ and this is indicated in Fig. \ref{fig:13}-b.

\begin{figure}[H]
\centering\includegraphics[width=1\linewidth]{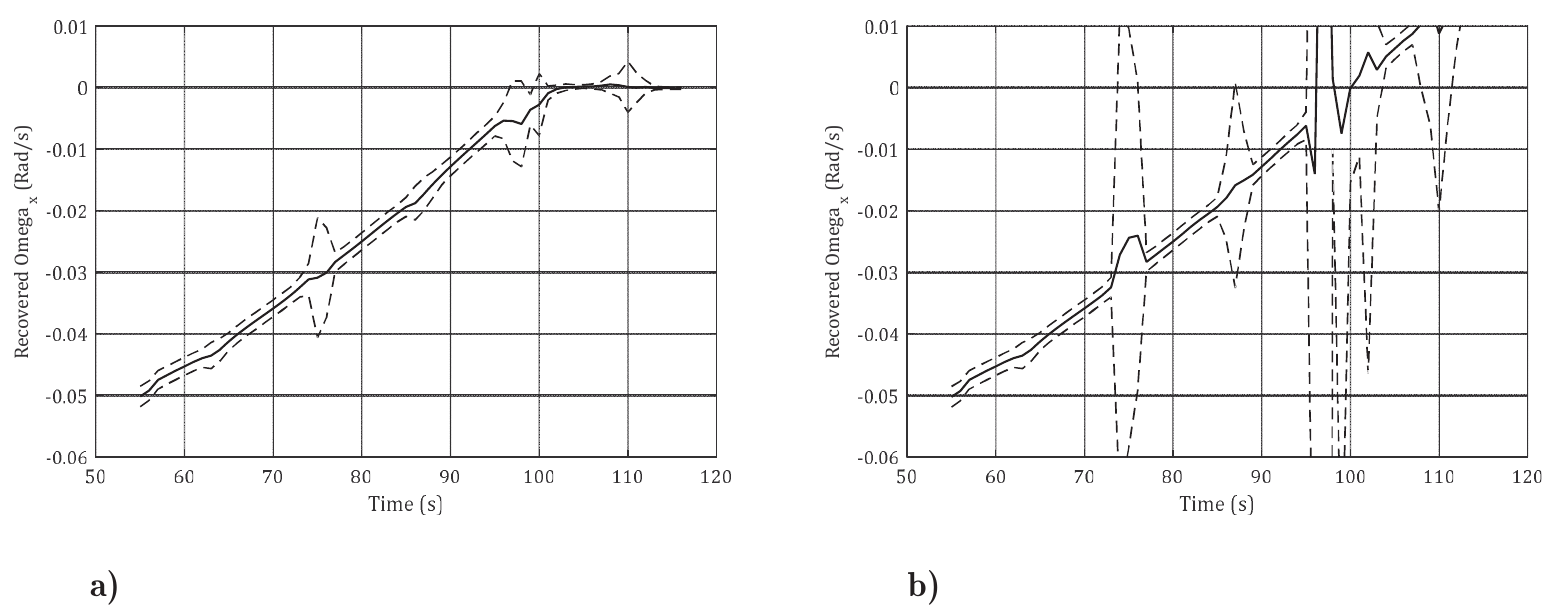}
\caption{Recovered $\omega_y$ for scenario (II) with $y$-gyroscope possibilities: a) scale factor; b) bias.}
\label{fig:12}
\end{figure}

\begin{figure}[H]
\centering\includegraphics[width=1\linewidth]{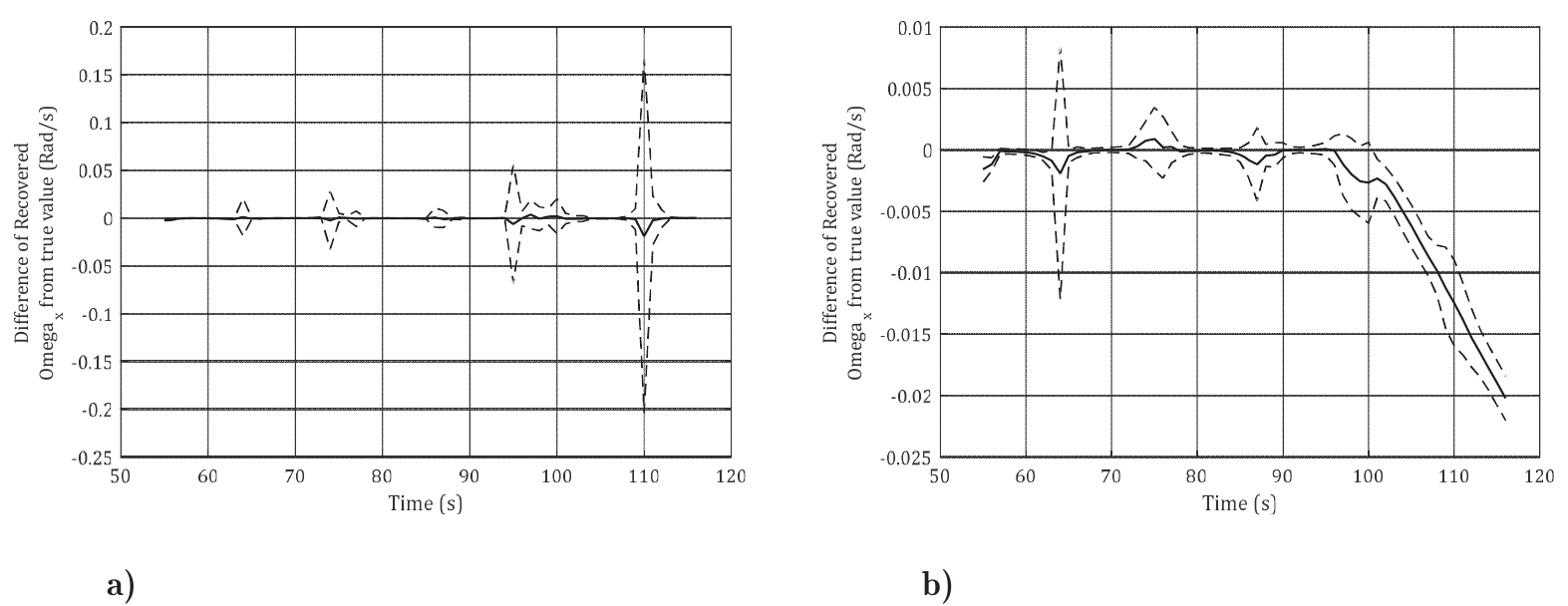}
\caption{The recovered error of $\omega_y$ for scenario (II) with $y$-gyroscope possibilities: a) scale factor; b) bias.}
\label{fig:13}
\end{figure}

Fig. \ref{fig:14} shows a comparison between different scale factors fault identification in $x$-gyroscope for scenario (I). As the scale factor approaches unity the estimation accuracy decreases. For $s_x=0$, we have the best estimation which standard deviations approaches zero at most of the times. This is because the random noise affects the scale factor identification around unity (normal condition).

\begin{figure}[H]
\centering\includegraphics[width=0.9\linewidth]{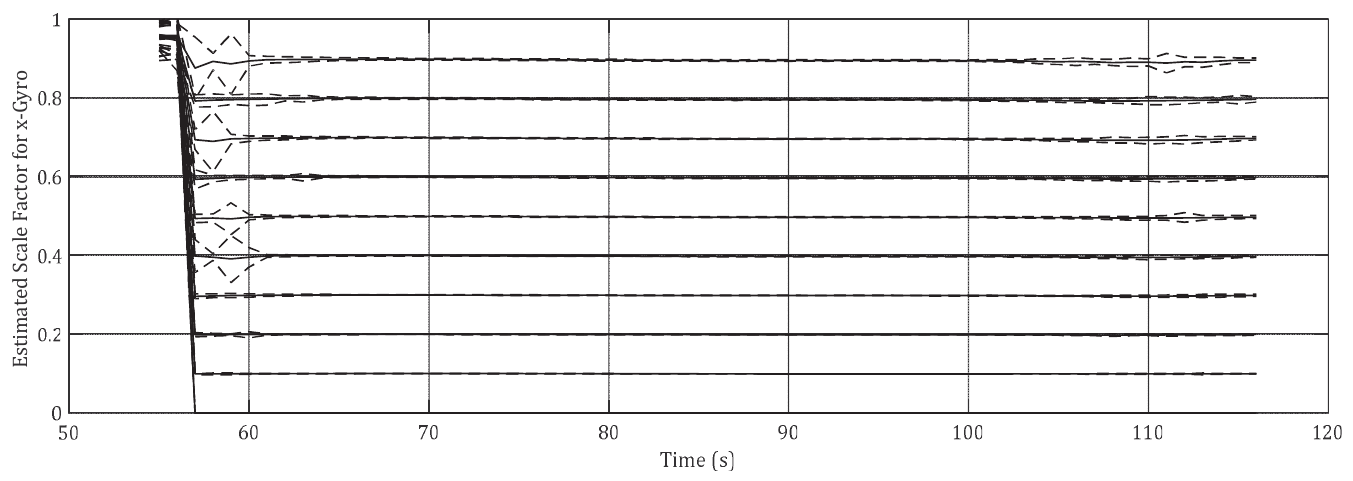}
\caption{A comparison between different scale factor estimations ($s_x=0:0.1:0.9$) for scenario (I).}
\label{fig:14}
\end{figure}

\subsection{Bias Faults}

Estimation of the bias would not fail as the angular velocities approach zero. The method is same as described for scale factor in previous section. Fig. \ref{fig:15} shows the estimated mean value and standard deviation of $x$-gyroscope biases ($b_x=0.1$, $1$, and $10\text{ deg/s}$) in logarithmic scale for simulation of scenario (II) supposing that the $x$-gyroscope bias is activated at $t=56\text{ s}$. Obviously the bias estimation performance increases as the bias value increases. This way, the bias estimation loses its superiority as the value of the bias approaches zero. Furthermore, as the parameter $A$ (for $x$-gyro estimation) approaches zero, the estimation singularity arises. Fig. \ref{fig:16} shows the time history of the parameter $A$ and $\pm1\sigma$ (its standard deviation). This parameter approaches zero for several times where three of them are catastrophic.

\begin{figure}[H]
\centering\includegraphics[width=0.9\linewidth]{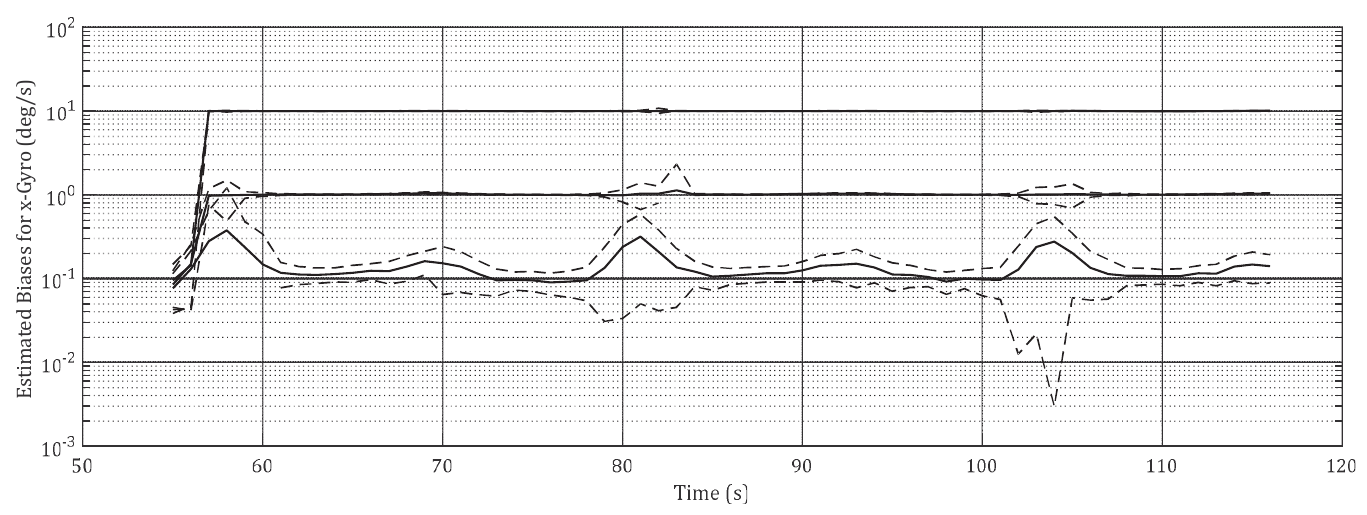}
\caption{A comparison between different values of biases ($b_x=0.1$, $1$ and $10\text{ deg/s}$) for scenario (II).}
\label{fig:15}
\end{figure}

\begin{figure}[H]
\centering\includegraphics[width=0.9\linewidth]{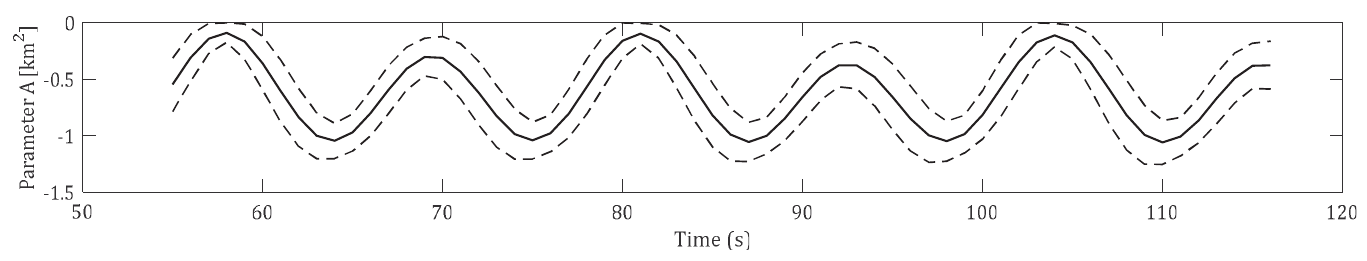}
\caption{The variation of the mean value and its deviation for the parameter $A$ for scenario (II).}
\label{fig:16}
\end{figure}

Taking different initial attitude characteristics and defining finer scenarios may prevent the singularity. However, as the results of Fig. \ref{fig:17} suggests, the recovered $\Psi$ functions even for $b_x=0.1\text{ deg/s}$ can sufficiently lead to an isolation process.

\begin{figure}[H]
\centering\includegraphics[width=1\linewidth]{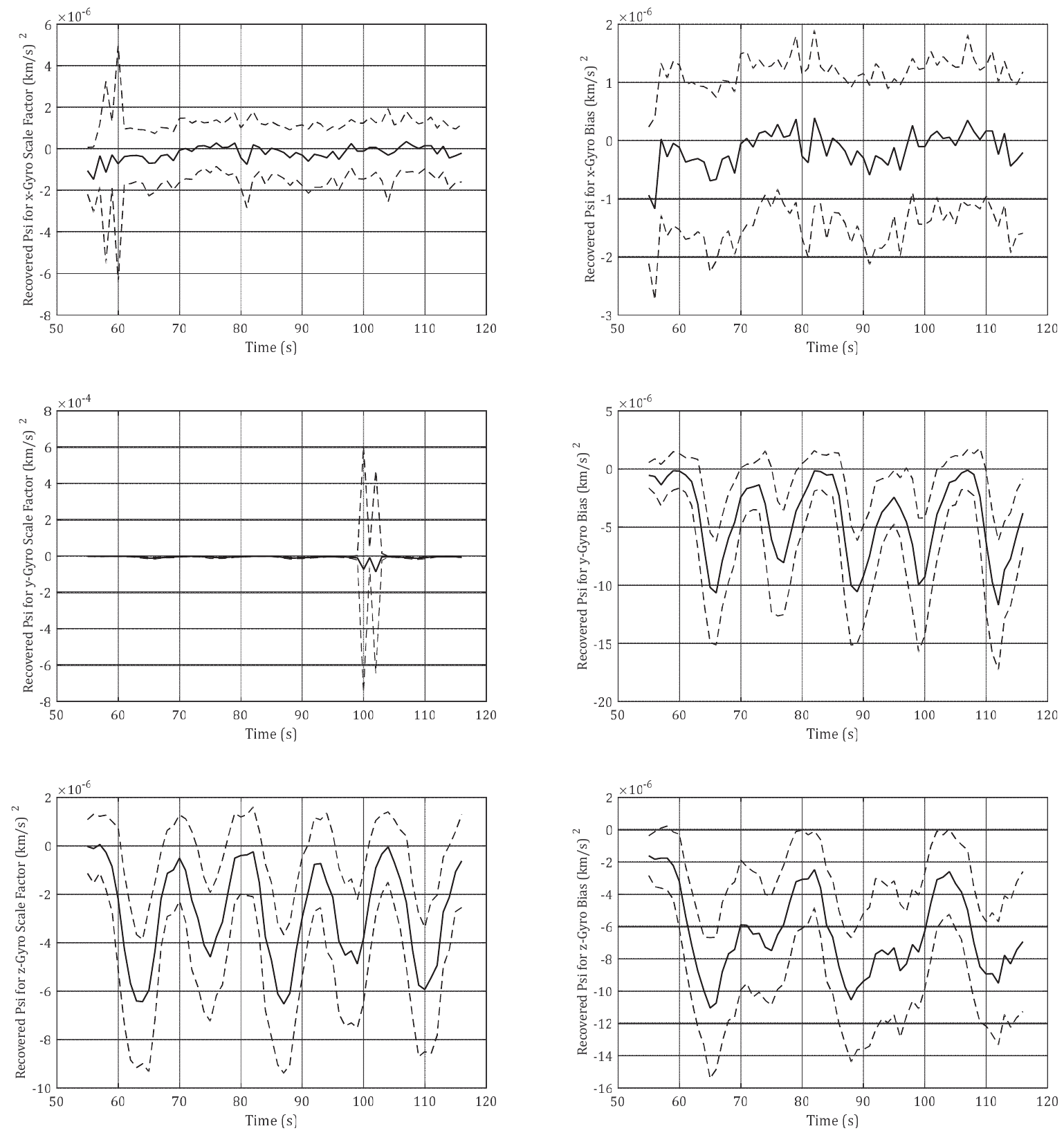}
\caption{Recovered $\Psi$ function of each possibility for scenario (II) with $b_x=0.1$.}
\label{fig:17}
\end{figure}

According to Fig. \ref{fig:17} both $x$-gyroscope scale factor and bias assumptions have recovered $\Psi$ functions laid within the threshold. According to Proposition \ref{prop:1}, both assumptions must lead to acceptable recovered values of $\omega_x$. Fig. \ref{fig:18} shows the difference of recovered and true values of $\omega_x$ in both possibilities.

\begin{figure}[H]
\centering\includegraphics[width=1\linewidth]{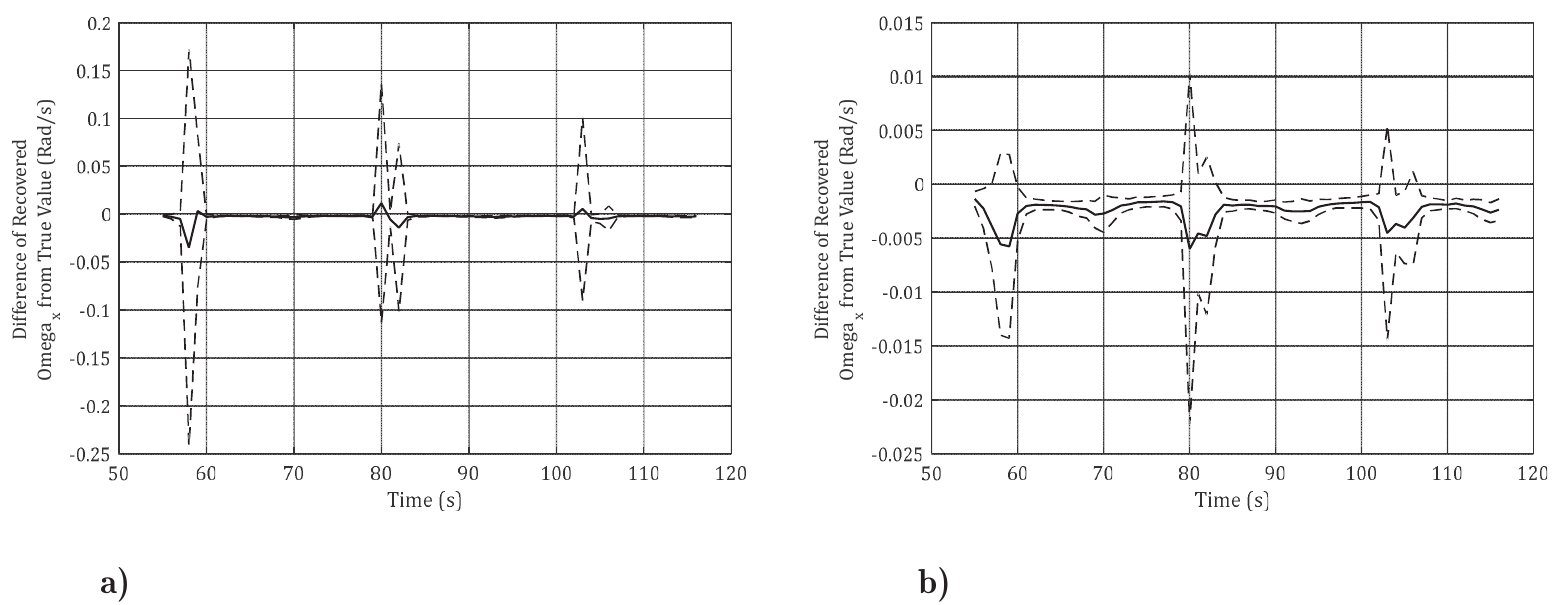}
\caption{The recovered error of $\omega_x$ for scenario (II) for $b_x=0.1$ with $x$-gyroscope possibilities: a) scale factor; b) bias.}
\label{fig:18}
\end{figure}

The most effective parameter in bias estimation is the value of parameter $A$ (for $x$-gyroscope case). As $A$ approaches zero (Fig. \ref{fig:16}) the estimation quality and subsequently the recovery performance decreases (for example at $t\simeq80\text{ s}$).

\section{Effect of Perturbations}
\label{S:7}

Perturbations can affect the satellite dynamics in an actual environment and these effects can be modeled as perturbed accelerations. Eq. \eqref{eq:2} has been derived by neglecting non-conservative perturbed forces, however, this section analyzes the effect of any neglected terms in the proposed FDI process, including conservative and non-conservative accelerations.

Variation of function $\widehat{\Psi}$ in the presence of perturbed forces can be considered equal to the variation of parameter $K$ as the other parameters are computed using relative distance measurements. $\delta K$ as a result of forcing model variation can be stated as:

\begin{equation}
\label{eq:27}
\delta K=\delta\widehat{\Psi}=\boldsymbol r_{SP}^T\delta\boldsymbol f
\end{equation}
\begin{equation}
\label{eq:28}
\delta\boldsymbol f={\boldsymbol a_p}_S-{\boldsymbol a_p}_P
\end{equation}
in which, ${\boldsymbol a_p}_S$ and ${\boldsymbol a_p}_P$ are vectors of perturbed accelerations. Returning to Eq. \eqref{eq:25} and considering $s_x<1$,
\begin{equation}
\label{eq:29}
\delta s_x=\frac{s_x(1-s_x)}{2\widetilde{\Psi}}\delta\widetilde{\Psi}
\end{equation}
Thus, the variation of $s_x$ can be stated as:
\begin{equation}
\label{eq:30}
\delta s_x=\frac{s_x(1-s_x)}{2\widetilde{\Psi}}\boldsymbol r_{SP}^T({\boldsymbol a_p}_S-{\boldsymbol a_p}_P)
\end{equation}
The maximum value for $\delta s_x$ with respect to the perturbations using Eq. \eqref{eq:16} is as follows:
\begin{equation}
\label{eq:31}
\max|\delta s_x|=\frac{1}{2r_{SP}\omega_x^2}\frac{s_x}{(1-s_x)}\max({a_p}_S+{a_p}_P)
\end{equation}
where ${a_p}_S$ and ${a_p}_P$ are the acceleration magnitudes. The most effective perturbation on the satellite around the Earth is the second zonal harmonics, known as $J_2$ perturbation.

The perturbation acceleration of primary satellite due to $J_2$ in the inertial frame of reference, using Cartesian coordinates are \citep{Curtis43}
\begin{equation}
\label{eq:32}
\begin{split}
{\boldsymbol a_p}_S=-\left[\frac{{\partial\Phi_P}_{J_2}}{{\partial r_P}_x}\boldsymbol i \quad \frac{{\partial\Phi_P}_{J_2}}{{\partial r_P}_y}\boldsymbol j \quad \frac{{\partial\Phi_P}_{J_2}}{{\partial r_P}_z}\boldsymbol k \right] \\
=-\frac{3}{2}\frac{\mu J_2R_e^2}{r_P^5}\left\{\begin{matrix}
\left[1-5\left(\frac{{r_P}_z}{r_P}\right)^2\right]{r_P}_x \\
\left[1-5\left(\frac{{r_P}_z}{r_P}\right)^2\right]{r_P}_y \\
\left[3-5\left(\frac{{r_P}_z}{r_P}\right)^2\right]{r_P}_z
\end{matrix}\right\}
\end{split}
\end{equation}
Therefore, ${a_p}_S$ can be found as:
\begin{equation}
\label{eq:33}
{a_p}_S=\frac{3}{2}\mu J_2R_e^2\frac{\sqrt{(1-\cos^2{\phi_P}_z)^2+4\cos^4{\phi_P}_z}}{r_P^6}
\end{equation}
where
$$\cos{\phi_P}_z=\frac{{r_P}_z}{r_P}$$
is the co-latitude cosine of the primary satellite position. The maximum value of ${a_p}_S$ is for $\cos^2{\phi_P}_z=1$. After repeating the same steps for the secondary satellite, the maximum value of ${a_p}_S+{a_p}_P$ is computed as:
\begin{equation}
\label{eq:34}
\max\left({a_p}_S+{a_p}_P\right)=3\mu J_2R_e^2\left[\frac{1}{(r_P^2+r_{SP}^2)^2}+\frac{1}{r_P^4}\right]
\end{equation}
With a very close approximation of $r_P^2+r_{SP}^2\simeq r_P^2$, the maximum value of $\delta s_x$ would be:
\begin{equation}
\label{eq:35}
\max|\delta s_x|=\frac{3}{r_{SP}\omega_x^2}\frac{s_x}{1-s_x}\frac{\mu J_2R_e^2}{r_P^4}
\end{equation}

From Eq. \eqref{eq:35} $\max|\delta s_x|/s_x=3\mu J_2 R_e^2/[r_{SP}r_P^4\omega_x^2(1-s_x)]$. Suppose that the desired value of $\max|\delta s_x|/s_x$ should be less than $\alpha\in(0,1)$, i.e., $\sup_{s_x}(max|\delta s_x|/s_x)=\alpha$. Let $s_x^+=1−3\mu J_2R_e^2/(\alpha r_{SP}\omega^2r_P^4)$ and $\mathbb{S}_x=\{s_x\in\mathbb{R}|0<s_x<s_x^+\}$, then it is obvious that the relative error due to $J_2$ perturbation is less than $\alpha$ if $s_x\in\mathbb{S}_x$. As an example for $\omega_x=0.055\text{ rad/s}$ (at $t\simeq60\text{ s}$ of scenario II), $s_x^+=1−1.4\times10^{-3}/\alpha$. So, for $\alpha\in(0,1.4\times10^{-3}]$, $\mathbb{S}_x=\varnothing$ and for $\alpha\in(1.4\times10^{-3},1]$, $\mathbb{S}_x\neq\varnothing$. Fig. \ref{fig:19} shows $s_x^+$ vs. $\alpha$; it shows that the relative error due to effect of $J_2$ perturbation on the scale factor is less than $\alpha=10^{-1}$ for a wide range of scale factors.

Similarly, by using Eq. \eqref{eq:26} the variation of the estimated bias is
\begin{equation}
\label{eq:36}
\delta b_x=\pm\frac{1}{2Ab_x}\delta\widetilde{\Psi}
\end{equation}
The maximum value of Eq. \eqref{eq:36} would be
\begin{equation}
\label{eq:37}
\max|\delta b_x|=\frac{1}{2b_x}\max\left({a_p}_S+{a_p}_P\right)
\end{equation}
In the same manner of scale factor analysis, by substituting Eq. \eqref{eq:34} into \eqref{eq:37} and assuming $r_P^2+r_{SP}^2\simeq r_P^2$, the following maximum value of bias estimation error is found:
\begin{equation}
\label{eq:38}
\max|\delta b_x|=\frac{3}{b_x}\frac{\mu J_2R_e^2}{r_P^4}
\end{equation}

From Eq. \eqref{eq:35} $\max|\delta b_x|/b_x=3\mu J_2R_e^2/(r_P^4b_x^2)$. Similar to the case of scale factor, suppose the desired value of $\max|\delta b_x|/b_x$ should be less than $\beta\in(0,1)$, i.e., $\sup_{b_x}(\max|\delta b_x|/b_x)=\beta$. Let $b_x^+=\sqrt{3\mu J_2R_e^2/(\beta r_P^4)}$ and $\mathbb{B}_x=\{b_x\in\mathbb{R}|b_x^+<b_x\}$, then it is obvious that the relative error due to $J_2$ perturbation is less than $\beta$ if $b_x\in\mathbb{B}_x$. As an example, at $t\simeq60\text{ s}$ of scenario II, $b_x^+=\sqrt{4.15\times10^{-6}/\beta}$. So, for $\beta\in(0,1]$, $\mathbb{B}_x\neq\varnothing$. Fig. \ref{fig:19} shows $b_x^+$ versus $\beta$; it shows that the relative error due to effect of $J_2$ perturbation on the bias is more than $\beta=10^{-1}$ for a small range of biases, namely less than $b_x\simeq0.006$.

\begin{figure}[H]
\centering\includegraphics[width=1\linewidth]{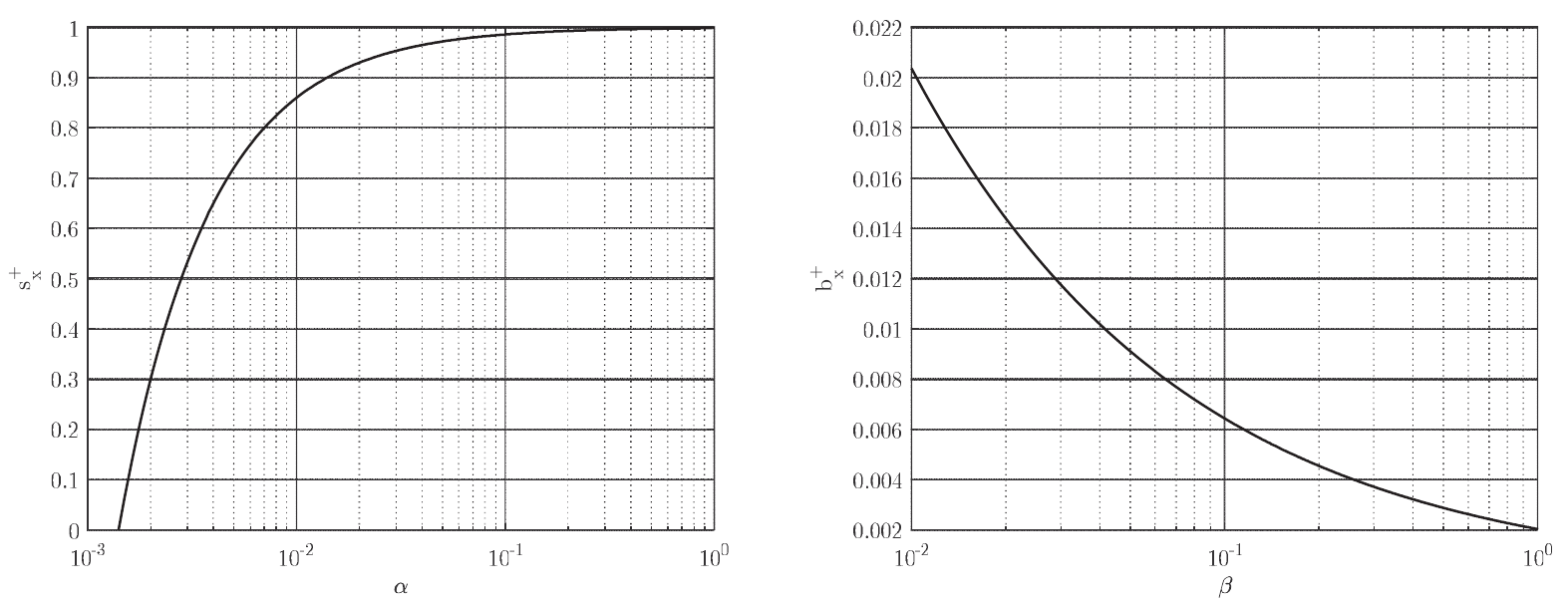}
\caption{Effect of $J_2$ perturbation on the maximum relative error of scale factor (left), and bias (right).}
\label{fig:19}
\end{figure}

Since drag is another affecting perturbation for a low Earth orbiting satellite, a model of ${a_d}_i=q_iS_i{C_d}_i/m_i$ is considered for the $i$th satellite ($i\triangleq P$ or $S$ for primary and secondary satellites), where $q_i$ is the dynamic pressure, $m_i$ is the mass, $S_i$ is the cross-section area, and ${C_d}_i$ is the drag coefficient of the corresponding satellite. Substituting the drag perturbation acceleration in Eqs. \eqref{eq:30} and \eqref{eq:36}, considering two satellites in circular orbits with equal masses $m$, drag coefficients $C_d$, and areas $S$, the maximum error in estimation of scale factor and bias are:
\begin{equation}
\label{eq:39}
\begin{split}
\max|\delta s_x|=\frac{\mu}{4r_{SP}\omega_x^2}\frac{s_x}{1-s_x}\frac{SC_d\rho_0}{m}\exp\left(\frac{h_0+R_e}{H}\right)\times \\
\left|\exp\left(-\frac{r_P+r_{SP}}{H}\right)\frac{1}{r_P+r_{SP}}-\exp\left(-\frac{r_p}{H}\right)\frac{1}{r_P}\right|
\end{split}
\end{equation}

From these equations, it can be seen that for the simulated system in this paper with $C_d=2.5$, $m=4.1\text{ kg}$, $S=1\text{ m}^2$, and relative distances less than $50\text{ km}$, the drag effect is very small (almost $\max|\delta s_x|/s_x=1.1631\times10^{-10}/(1-s_x)$ and $\max|\delta b_x|/b_x=1.7592\times10^{-11}/b_x^2$) in comparison to the two body effect. Thus, it can be concluded that for FF satellites with short relative distances the effect of other terms can be simply disregarded from analysis for the range of scale factors and biases above selected threshold.

\begin{remark}
As it can be concluded from Eqs. \eqref{eq:31} and \eqref{eq:37}, the effect of perturbations tend to infinity if scale factor approaches unity and/or bias approaches zero. So, it might be supposed that the proposed method fails for slight faults due to effect of perturbations. However, an appropriate selection of threshold value can prevent the method from failure. The thresholds are defined on $\widetilde{\Psi}$, while $\delta\widetilde{\Psi}$ is independent of scale factors and biases. Hence, a proper threshold can be designed in order to detect and isolate those faults that are in the sets of $\mathbb{S}_x$ and $\mathbb{B}_x$ defined above. Therefore, even in presence of perturbations, the proposed fault detection and isolation technique is applicable and the order of the accuracy is adjustable through the value of threshold.
\end{remark}

\section{Conclusions}
\label{S:8}

This paper has proposed and analyzed a relative dynamic fault detection and isolation approach for a pair of Earth-orbiting satellites in a two body gravitational model. The algorithm makes use of the relative position vector in the body coordinates of primary satellite and obtains the relative velocity and acceleration vectors by a numerical differentiation method. These data including the primary satellite position, constructs a scalar function of primary satellite gyroscopes outputs. Deviation of this function from the threshold is used as a measure of gyroscope fault and its behavior after fault detection is utilized for fault isolation purposes. The algorithm does not need any knowledge of secondary satellite orbit or attitude. Moreover, because the satellites in formation already have the relative distance measurement sensors, no additional sensor is required for the proposed algorithm.

The algorithm performance strongly depends on the relative sensor noises and the numerical differentiation method. The isolation accuracy decreases in special conditions that the secondary satellite nearly lays in the same direction of faulty gyroscope. Accuracy also decreases for scale factor estimation as the angular velocity approaches zero in the corresponding direction. The effect of perturbations depends on the order of faults. In a conventional condition the maximum relative error can be in the order of $10^{-3}$ for scale factor and bias estimations.

\bibliography{mybibfile}

\end{document}